\documentclass[11pt,a4paper,reqno, draft]{amsart}
\usepackage{amsmath, amssymb, latexsym, enumerate,  graphicx,tikz, stmaryrd, eufrak,enumitem,MnSymbol}

\usepackage{thmtools}

\usepackage[pdftex]{hyperref}
\usepackage{cleveref}
\usepackage{bbm}
\usepackage{cases}
\usepackage{array}
\usepackage{tikz-cd}
\usepackage[all]{xy}

\usepackage[hmarginratio={1:1},vmarginratio={1:1},lmargin=80.0pt,tmargin=90.0pt]{geometry}

\usepackage{tikz}
\tikzset{anchorbase/.style={baseline={([yshift=-0.5ex]current bounding box.center)}}}
\usetikzlibrary{decorations.markings}
\usetikzlibrary{decorations.pathreplacing}
\usetikzlibrary{arrows,shapes,positioning,backgrounds}
\tikzstyle directed=[postaction={decorate,decoration={markings,
    mark=at position #1 with {\arrow{>}}}}]
\tikzstyle rdirected=[postaction={decorate,decoration={markings,
    mark=at position #1 with {\arrow{<}}}}]
    
\usepackage{graphicx}
\usepackage[vcentermath]{youngtab}
\usepackage{nicefrac}

 \newlength{\baseunit}               % the basic unit length
 \newcount{\numlines}                % depth of picture (in number of lines)
\newtheorem{theorem}[subsubsection]{Theorem}
\newtheorem{lemma}[theorem]{Lemma}
\newtheorem{prop}[theorem]{Proposition}
\newtheorem{corollary}[subsubsection]{Corollary}

\theoremstyle{definition}
\newtheorem{definition}[subsubsection]{Definition}

\newtheorem{remark}[theorem]{Remark}

\newtheorem{example}[subsubsection]{Example}

\newtheorem{question}[theorem]{Question}
\newtheorem{analogy}[theorem]{Analogy}

\newtheorem{thmA}{Theorem}

\newcommand{\PSh}{\mathrm{PSh}}
\newcommand{\Sh}{\mathrm{Sh}}

\newcommand{\normalt}{t}

\newcommand{\cJ}{\mathcal{J}}
\newcommand{\cK}{\mathcal{K}}
\newcommand{\cT}{\mathcal{T}}
\newcommand{\Co}{\mathcal{C}o}

\newcommand{\ev}{\mathrm{ev}}
\newcommand{\co}{\mathrm{co}}

\newcommand{\bA}{\mathbf{A}}
\newcommand{\bB}{\mathbf{B}}
\newcommand{\bC}{\mathbf{C}}
\newcommand{\bD}{\mathbf{D}}
\newcommand{\bT}{\mathbf{T}}
\newcommand{\bV}{\mathbf{V}}
\newcommand{\bU}{\mathbf{U}}

\newcommand{\svecc}{\mathsf{svec}}

\newcommand{\cE}{\mathcal{E}}
\newcommand{\charr}{\mathrm{char}}

\newcommand{\Set}{\mathsf{Set}}
\newcommand{\Ab}{\mathsf{Ab}}

\newcommand{\Mod}{\mathsf{Mod}}
\newcommand{\Rep}{\mathsf{Rep}}
\newcommand{\Tens}{\mathsf{Tens}}
\newcommand{\GL}{\mathrm{GL}}
\newcommand{\SG}{\mathrm{S}}
\newcommand{\OO}{\mathrm{O}}
\newcommand{\OSp}{\mathrm{OSp}}
\newcommand{\unit}{{\mathbbm{1}}}
\newcommand{\tto}{\twoheadrightarrow}

\newcommand{\mN}{\mathbb{N}}
\newcommand{\mZ}{\mathbb{Z}}
\newcommand{\mC}{\mathbb{C}}

\newcommand{\End}{\mathrm{End}}
\newcommand{\Ob}{\mathrm{Ob}}
\newcommand{\Ext}{\mathrm{Ext}}

\newcommand{\Hom}{\mathrm{Hom}}

\newcommand{\id}{\mathrm{id}}
\newcommand{\Nat}{\mathrm{Nat}}
\newcommand{\St}{\mathrm{St}}

\newcommand{\op}{\mathrm{op}}
\newcommand{\Ind}{\mathrm{Ind}}

\newcommand{\Vecc}{\mathsf{Vec}}

\newcommand{\Tilt}{\mathsf{Tilt}}

\begin{document}
\title[Monoidal abelian envelopes]{Monoidal abelian envelopes}
\author{Kevin Coulembier}

%\date{\today}
%\thanks{.}


\thanks{School of Mathematics and Statistics, University of Sydney, NSW 2006, Australia\\
kevin.coulembier@sydney.edu.au\hspace{5mm} 0000-0003-0996-3965}

\begin{abstract}
We prove a constructive existence theorem for abelian envelopes of non-abelian monoidal categories. This establishes a new tool for the construction of tensor categories. As an example we obtain new proofs for the existence of several universal tensor categories as conjectured by Deligne. Another example constructs interesting tensor categories in positive characteristic via tilting modules for $SL_2$.
\end{abstract}

\maketitle

%\tableofcontents

\section*{Introduction}
Fix a field $k$.
A $k$-linear symmetric rigid monoidal karoubian category in which the endomorphisms of the tensor identity $\unit$ constitute $k$ will be called a `pseudo-tensor category'. When the category is abelian, it is called a `tensor category', following~\cite{Del90}. The canonical example of the latter is the category of algebraic representations of an affine group scheme over $k$. It is often easy to construct specific examples of {\em pseudo}-tensor categories, for instance diagrammatically or via generators and relations. On the other hand, constructing tensor categories with certain requested properties is typically more challenging. In many recent constructions of important new tensor categories, see \cite{BE, CO, CEH, Deligne, EHS}, the desired tensor categories happen to be `abelian envelopes' of straightforward pseudo-tensor categories. We review these examples below via applications of our main result.

%One technique The notion of `abelian envelopes' of non-abelian monoidal categories has recently appeared in several instances in the literature. We will explain below some of the relevance of these envelopes via two applications of our main results.

A tensor category is the abelian envelope of a pseudo-tensor subcategory if every faithful tensor functor from the subcategory to a tensor category lifts to an exact tensor functor out of the original category. Not every pseudo-tensor category admits an abelian envelope. A classical example is given in \cite[\S 5.8]{Deligne} and we will give an example of a different nature below. A powerful `recognition theorem' for abelian envelopes was obtained in \cite{EHS}. However, the construction of abelian envelopes in \cite{BE, CO, EHS} drew from a rich variety of different methods, rather than some standard approach, and moreover at present there is no `existence theorem' in the literature for abelian envelopes.

The latter is precisely the aim of the current paper. We derive sufficient internal conditions on a pseudo-tensor category for its abelian envelope to exist, along with a unifying construction of the envelope. We apply this to recover old and construct new abelian envelopes. 
To state our main theorem, we call an object $X$ with dual $X^\vee$ in a pseudo-tensor category $\bD$ `strongly faithful' if the evaluation $X^\vee\otimes X\to \unit$ is the coequaliser of the two evaluation morphisms $X^\vee\otimes X\otimes X^\vee\otimes X\rightrightarrows X^\vee\otimes X$. We show that this is equivalent to the property that $X\otimes-:\bD\to\bD$ reflects all kernels and cokernels in $\bD$.

\begin{thmA}\label{ThmA}
If for every morphism $f$ in $\bD$ there exists a strongly faithful $X\in\bD$ for which $X\otimes f$ is split, then  $\bD$ admits an abelian envelope $\bT$. Moreover, the ind-completion $\Ind\bT$ is tensor equivalent to the category
$\Sh\bD$ of all presheaves $\bD^{\op}\to\Vecc_k$ which send the sequences
$$D\otimes X^\vee\otimes X\otimes X^\vee\otimes X\to D\otimes X^\vee\otimes X\to D\to 0,$$
for all $D\in\bD$ and strongly faithful  $X\in\bD$, to exact sequences in $\Vecc_k$.
\end{thmA}

A slightly more general version of this is proved in Theorem~\ref{Thm}. In the following sense Theorem~\ref{ThmA} cannot be improved. In Lemma~\ref{FunnyEx} we provide a category $\bD$ where all assumptions are satisfied but with `strongly faithful' replaced by the weaker `faithful' in ordinary sense ($X\otimes -$ is faithful) and which does not admit an abelian envelope. We also demonstrate that the recognition theorem from \cite{EHS} can be derived from Theorem~\ref{ThmA}. In particular, Theorem~\ref{ThmA} gives an explicit construction of the abelian envelope in all cases where one might apply said recognition theorem.

Note that, under the assumptions in Theorem~\ref{ThmA}, one can prove that every non-zero object in $\bD$ is strongly faithful, in particular the definition of $\Sh\bD$ can then be adjusted. However,  we demonstrate that $\Sh\bD$ as defined above is always (without the splitting condition in Theorem~\ref{ThmA}) the category of sheaves with respect to some $k$-linear Grothendieck topology on $\bD$. This shows that $\Sh\bD$ is always a symmetric closed monoidal Grothendieck category. We also observe that whenever $\Sh\bD$ is the ind-completion of some tensor category, the latter must be the abelian envelope of $\bD$. Moreover, we determine an intrinsic criterion for when $\Sh\bD$ is the ind-completion of a tensor category.

Simultaneously and independently, Benson, Etingof and Ostrik have obtained related results in \cite{BEO}. On the one hand, the scope {\it loc. cit.} is more general in the sense that it does not require braidings and it also considers some analogue of envelopes in which the subcategory is not full. On the other hand, \cite{BEO} is restricted to tensor categories which have enough projective objects, which for instance does not include the ones in \cite{EHS}.

\subsection*{Application I: Deligne's universal monoidal categories} Let $k$ be a field of characteristic $0$.
In \cite{Deligne}, Deligne introduced three 1-parameter families of universal pseudo-tensor categories $[\SG_t,k]$, $[\GL_t,k]$ and $[\OO_t,k]$, for $t\in k$, and embedded them into tensor categories. He also formulated conjectures about the universality of the latter. As observed in \cite{CO, EHS}, the 
conjectures can be reformulated, via the tannakian formalism of \cite{Del90}, into the existence of abelian envelopes.

These conjectures were proved for $[\SG_t,k]$ in \cite{CO} and for $[\GL_t,k]$ in \cite{EHS}. In \cite{CO} the envelope is constructed via a suitable t-structure on the homotopy category $K^b([\SG_t,k])$ and in \cite{EHS} the envelope of $[\GL_t,k]$ is realised as a suitable limit of truncations of representation categories of general linear supergroups of growing rank.

Since Theorem~\ref{ThmA} applies to $[\SG_t,k]$, $[\GL_t,k]$ and $[\OO_t,k]$, it gives a new and unifying proof and construction of all the abelian envelopes, so of all corresponding universal tensor categories. Moreover, we do not require $\bar{k}=k$, contrary to \cite{EHS}. The construction of the abelian envelope of $[\OO_t,k]$ is new, although one would expect that the methods from \cite{EHS} can be extended to this case.

Yet another construction of the abelian envelope of $[\GL_t,\mC]$, described in \cite{Harman}, realises it inside an ultraproduct $\prod_{\mathcal{U}}[\GL_{t_i},\overline{\mathbb{F}}_{p_i}]$. However, recognising the tensor category inside the product as the abelian envelope requires the knowledge of the existence of the latter (as proved first in \cite{EHS}).

\subsection*{Application II: Tensor categories in positive characteristic}
The structure theory of tensor categories over fields of positive characteristic is in full development, see for instance \cite{BE, BEO, Tann, EO, EG, Ostrik}. An important tool developed in \cite{Tann, EO, Ostrik} is the `Frobenius twist' in arbitrary tensor categories. In \cite{BE} a family of tensor categories in characteristic~2 was constructed in which this functor is {\em not} exact. One way to interpret those categories, is as the abelian envelopes of the monoidal quotients of the pseudo-tensor category $\Tilt SL_2$ of tilting modules of the reductive group $SL_2$. We will show that these quotients also admit abelian envelopes when $p>2$ by application of Theorem~\ref{ThmA}.

These envelopes are also constructed independently in \cite{BEO}, and studied in full detail there. In particular, they provide the first examples of tensor categories for $p>2$ on which the Frobenius twist is not exact.

\subsection*{Structure of the paper}
In Section~\ref{Prel} we recall the necessary background. In Section~\ref{SecSplit} we introduce and study the notions of strongly faithful objects and monoidal splitting of morphisms. As an application, we show that the conditions in Theorem~\ref{ThmA} are satisfied for $[\GL_t,k]$ and $[\OO_t,k]$. In Section~\ref{SecSh} we study the category $\Sh\bD$. In Section~\ref{SecMain} we then apply all the above to prove Theorem~\ref{ThmA} and apply it to the above examples.

In Appendix~\ref{AppTop} we recall the notions of Grothendieck topologies and sheaves on $k$-linear sites. An alternative approach to the methods in Section~\ref{SecSh} would be to argue that our set-up allows to apply a general theory developed in \cite{Sch2} by Sch\"appi. Since our case is rather specific, it is more transparent to use a direct approach, but in order to highlight this connection we also recall some results from \cite{Sch2} in Appendix~\ref{AppTop}.

\section{Preliminaries}\label{Prel} 
We set $\mN=\{0,1,2,\ldots\}$. Throughout the paper we let $k$ denote an arbitrary field, unless further specified.

\subsection{Exactness and split morphisms}
Let $\bA$ be a preadditive category.
\subsubsection{} We denote by $\Xi=\Xi(\bA)$, the class of all exact sequenes
\begin{equation}\label{rexeq}X_2\stackrel{p}{\to} X_1\stackrel{q}{\to} X_0\to 0\end{equation}
in $\bA$. That is, all sequences \eqref{rexeq} where $q$ is the cokernel of $p$, which is equivalent to
$$0\to \bA(X_0,A)\xrightarrow{-\circ q} \bA(X_1,A)\xrightarrow{-\circ p} \bA(X_2,A)$$
being exact in $\Ab$ for each $A\in\bA$. %We also say that $X\to Y\to 0$ is exact if and only if the associated $X\to Y\to 0\to 0$ is exact.

\subsubsection{} A morphism $f:X\to Y$ in $\bA$ is {\bf split} if there exists $g:Y\to X$ such that $f\circ g\circ f=f$. Note that this implies that $f\circ g$ and $g\circ f$ are idempotents. If $\bA$ is Karoubi (idempotent complete) it thus follows that $f$ is split if and only if we have $X\simeq A\oplus X_0$ and $Y\simeq A\oplus Y_0$ and $f$ is the composition of these isomorphisms with $(\id_A,0)$.

\subsection{Symmetric monoidal categories}
Let $K$ be a commutative ring.
\subsubsection{} By a $K$-linear symmetric monoidal category $(\bC,\otimes, \unit,\sigma)$, we mean a monoidal category $(\bC,\otimes,\unit)$ with a symmetric braiding $\sigma$ with a fixed $K$-linear structure on $\bC$ for which $-\otimes -$ is $K$-linear in each variable. As is customary, we suppress the associativity constraints and unitors from all notation. Correspondingly we do not place brackets in iterated tensor products. Furthermore, in order to keep long expressions legible, the functor $X\otimes-$, for $X\in\bC$, will sometimes be shortened to $X-$. So we might write $XY$ or $Xf$ for an object $Y$ or morphism $f$ in $\bC$.

\subsubsection{} A {\bf tensor functor} between two $K$-linear symmetric monoidal categories is a $K$-linear symmetric monoidal functor. Usually we will denote the tensor functor simply by the underlying functor.  For two $K$-linear symmetric monoidal categories $(\bC,\otimes, \unit,\sigma)$ and $(\bC',\otimes', \unit',\sigma')$, we denote by $\Tens(\bC,\bC')$ the category of tensor functors $\bC\to\bC'$. 
A
{\bf tensor equivalence} is a tensor functor which is also an equivalence.

\subsubsection{} For a $K$-linear symmetric monoidal category $(\bC,\otimes, \unit,\sigma)$ and $X\in\bC$, a dual of $X$ is a triple $(X^\vee,\ev_X,\co_X)$ of an object $X^\vee\in\bC$ and morphisms $\ev_X:X^\vee\otimes X\to\unit$ and $\co_X:\unit\to X\otimes X^\vee$, such that
\begin{equation}\label{snake}\id_X=(X\otimes \ev_X)\circ (\co_X\otimes X)\quad\mbox{and}\quad \id_{X^\vee}=(\ev_X\otimes X^\vee)\circ (X^\vee\otimes \co_X).\end{equation}
An object which admits a dual is called {\bf rigid}. If every object in $\bC$ admits a dual, then $\bC$ is called rigid.
The dimension $\dim(X)\in\End(\unit)$ of a rigid object is given by $\ev_X\circ \sigma_{XX^\vee}\circ \co_X$.

\subsubsection{}
A tensor ideal $\cJ$ in a $K$-linear symmetric monoidal category $(\bC,\otimes, \unit,\sigma)$ is an assignment of $K$-submodules $\cJ(X,Y)\subset\bC(X,Y)$ for each $X,Y\in\bC$ such that the corresponding class of morphisms is closed under composing or taking the tensor product with any morphism in $\bC$. For a tensor ideal $\cJ$, the quotient category $\bC/\cJ$ has by definition the same objects as $\bC$ and as morphism sets the quotient $K$-modules $\bC(X,Y)/\cJ(X,Y)$. By construction, $\bC/\cJ$ is again $K$-linear symmetric monoidal, such that $\bC\to\bC/\cJ$ is a tensor functor. We can therefore alternatively define tensor ideals as the kernels of tensor functors.

\subsection{Pseudo-tensor categories}
Let $k$ be an arbitrary field.
\subsubsection{}\label{DefPseudo} A $k$-linear symmetric monoidal category $(\bD,\otimes,\unit,\sigma)$ is a {\bf pseudo-tensor category over $k$} if
\begin{enumerate}[label=(\roman*)]
\item $\bD$ is essentially small;
\item $k\to\End(\unit)$ is an isomorphism;
\item $(\bD,\otimes,\unit,\sigma)$ is rigid;
\item $\bD$ is pseudo-abelian (additive and Karoubi).
\end{enumerate}

A pseudo-tensor subcategory of such $\bD$ is a full monoidal subcategory closed under taking duals, direct sums and summands. It is thus again a pseudo-tensor category. The quotient of a pseudo-tensor category with respect to a non-trivial tensor ideal is again pseudo-tensor.

Occasionally we will encounter categories as above except that the field $k$ is replaced by some commutative ring $R$. We will use the same terminology `tensor category over $R$'.

If only (i)-(iii) are satisfied, we can take the pseudo-abelian envelope, see \cite[\S 1.2]{AK}, by formally adjoining direct sums and summands, to obtain a pseudo-tensor category.

\begin{remark}\label{RemSigma}
Let $\bD$ be a pseudo-tensor category and $\xi\in\Xi(\bD)$. For any $A\in\bD$, the sequence $A\otimes \xi$ is still exact, so $A\otimes\xi \in \Xi$, since $A\otimes-$ has a right adjoint $A^\vee\otimes-$.
\end{remark}

\subsubsection{}Following  \cite{Del90, Del02}, a {\bf tensor category over $k$} is a pseudo-tensor category which is abelian (i.e. assumption \ref{DefPseudo}(iv) is strengthened). In such a category, $\unit$ is automatically a simple object.
Following \cite{CEH, EHS}, we use the following terminology.

\begin{definition}\label{DefAbEnv}
For a pseudo-tensor category $\bD$ over $k$, a pair $(F,\bT)$ of a tensor category $\bT$ over $k$ and a faithful tensor functor $F:\bD\to\bT$ constitute an {\bf abelian envelope} of $\bD$ if for each tensor category $\bT_1/k$, composition with $F$ induces an equivalence 
$$\Tens^{ex}(\bT,\bT_1)\;\simeq\; \Tens^{faith}(\bD,\bT_1)$$
between the categories of exact (resp. faithful) tensor functors.
\end{definition}

We will indulge in the usual abuse of terminology, by referring to the tensor category $\bT$ of a pair $(F,\bT)$ as in \ref{DefAbEnv} as `{\em the} abelian envelope of $\bD$'. The use of the definite article is justified by obvious uniqueness up to equivalence.

\begin{remark}\label{RemEF}
%In \cite{Del90, Del02}, what we call `exact tensor functors' between tensor categories are simply called tensor functors. 

By \cite[Corollaire 2.10(ii)]{Del90} functors in $\Tens^{ex}(\bT,\bT_1)$ are automatically faithful, so composition with $F$ in Definition~\ref{DefAbEnv}, automatically lands in $\Tens^{faith}(\bD,\bT_1)$. Furthermore, \cite[Corollaire 2.10(i)]{Del90} shows that right exact functors in $\Tens(\bT,\bT_1)$ are automatically in $\Tens^{ex}(\bT,\bT_1)$.
\end{remark}

\subsubsection{}
 For a tensor category $\bT$, the ind-completion $\Ind\bT$ is canonically an abelian symmetric monoidal category such that $-\otimes-$ is exact (and cocontinuous) in each variable, see \cite[\S 7]{Del90}. Since $\bT$ is assumed to be essentially small, we can define $\Ind\bT$ also as the category of left exact functors $\bT^{\op}\to\Vecc$. The following lemma is a special case of a general result in \cite{CP}, but we prove it by a direct generalisation of the argument in \cite[\S 2.2]{Del02} for tensor categories with all objects of finite length.
\begin{lemma}\label{LemInd}
The subcategory of rigid objects in $\Ind\bT$ is equivalent to $\bT$.
\end{lemma}
\begin{proof}
Consider $X\in\Ind\bT$, which is a filtered colimit $\varinjlim_i X_i$ with $X_i\in \bT$, and label the defining morphisms as $a_i: X_i\to X$. 
If $X$ has a dual $X^\vee$, then the facts that $X\otimes X^\vee\simeq \varinjlim_i (X_i\otimes X^\vee)$ and that $\unit$ is compact imply that $\co_X$ can be written as a composition of some morphism $f:\unit\to X_i\otimes X^\vee$ and $a_i\otimes X^\vee$ for some $i$. Consequently we obtain a commutative diagram
$$\xymatrix{
X\ar[r]^-{\co_X X}\ar[rd]_-{f X}& X\otimes X^\vee\otimes X\ar[r]^-{X \ev_X}& X\\
&X_i\otimes X^\vee\otimes X\ar[r]^-{X_i \ev_X}\ar[u]_-{a_iX^\vee X}& X_i\ar[u]_-{a_i}.
}$$
By~\eqref{snake}, we can thus write $\id_X$ as a composition $X\to X_i\to X$. So $X$ is a direct summand of $X_i\in\bT$ and therefore isomorphic to an object in $\bT$.
\end{proof}

The following lemma is straightforward, but it will be useful to have it spelled out.
\begin{lemma}\label{LemTriv}
Consider a pseudo-tensor category $\bD$, with pseudo-tensor subcategory $\bD_0\subset\bD$ and $X\in \bD_0$. The full subcategory $\bD_1$ of objects $V\in\bD$ for which $V\otimes X\in\bD_0$ is a pseudo-tensor subcategory of $\bD$.
\end{lemma} 
\begin{proof}
That $\bD_1$ is closed under taking direct sums and summands follows from the corresponding property of $\bD_0$. Clearly $\unit\in\bD_1$. Now if $V,W\in\bD_1$, then by definition 
$$V\otimes X\otimes W\otimes X\in\bD_0\quad\Rightarrow \quad V\otimes W\otimes (X\otimes X^\vee\otimes X)\in\bD_0.$$
Since $X$ is a direct summand of $X\otimes X^\vee\otimes X$, it follows that $V\otimes W\otimes X$ is a direct summand of an object in $\bD_0$ and hence also in $\bD_0$. In conclusion $V\otimes W\in\bD_1$. That $\bD_1$ is closed under taking duals follows similarly.
\end{proof}

\subsubsection{} Consider a pseudo-tensor category $\bD$ over $k$ and a field extension $K/k$. The naive extension of scalars of $\bD$, see \cite[5.1.1]{AK}, is the $K$-linear category with same objects as $\bD$, but with morphism sets given by $K\otimes_k\bD(-,-)$.
We define $\bD_K$ as the Karoubi envelope of the naive extension of scalars. Note that in \cite[\S 5.3]{AK}, the notation $(\bD_K)^\sharp$ is used for what we call $\bD_K$. Now $\bD_K$ is canonically a pseudo-tensor category over $K$.

\subsection{Deligne's universal monoidal categories}
Fix a commutative ring $R$ and $\normalt\in R$.
\subsubsection{} Following \cite[\S 10]{Deligne}, we have the category $[\GL_\normalt,R]_0$, which is the free $R$-linear rigid symmetric monoidal category on one object $V_\normalt$ of dimension $\normalt$. 
Its objects are (up to isomorphism) tensor products of $V_\normalt$ and $V_\normalt^\vee$.

The pseudo-abelian envelope $[\GL_\normalt,R]$ is thus a pseudo-tensor category over $R$. 
By construction, every object $X$ in $[\GL_\normalt,R]$ is a direct summand of a direct sum of objects $\otimes^i V_\normalt\otimes \otimes^jV_\normalt^\vee$. We denote by $\deg X$ the minimal $d\in\mN$ such that $X$ is a direct summand of a direct sum of $\otimes^a V_\normalt \otimes \otimes^b V^\vee_\normalt$ with $a+b\le d$. By \cite[Th\'eor\`eme~10.5]{Deligne}, $[\GL_\normalt,k]$ is a semisimple tensor category when $\charr(k)=0$ and $\normalt\not\in\mZ$.

The following is a reformulation of \cite[Proposition~10.3]{Deligne}.
\begin{lemma}\label{LemUni}
Consider a pseudo-tensor category $\bD$ over $R$. Evaluation at $V_\normalt$ yields an equivalence between $\Tens([\GL_\normalt,R],\bD)$
and the groupoid of objects of dimension $\normalt$ in $\bD$ with their isomorphisms.
\end{lemma}

\subsubsection{}\label{FunGL}  Set $\bD:=[\GL_\normalt,k]$ for an algebraically closed field $k$ of characteristic zero, for $\normalt\in\mZ\subset k$. Consider the tensor category $\svecc$ of finite dimensional super vector spaces, see \cite[\S 1.4]{Del90}. Let $\GL(m|n)$ be the affine group scheme in $\svecc$ of automorphisms of the super space $k^{m|n}$ of even dimension $m$ and odd dimension $n$. As in \cite[0.3]{Del02}, we have the tensor category  $\Rep_k\GL(m|n)$ of its representations in $\svecc$ which restrict to the canonical $\mZ/2$-action along the homomorphism $\mZ/2\to \GL(m|n)$ defining the grading on $\GL(m|n)$.
As an application of Lemma~\ref{LemUni}, there exists a tensor functor
$$H_{m|n}:\bD\to \Rep_k\GL(m|n),\quad V_t\mapsto k^{m|n}$$
 for every $m,n\in\mN$ with $m-n=\normalt$.

 \begin{lemma}\label{FunGLLem}
Retain the notation of \ref{FunGL}.
 \begin{enumerate}[label=(\roman*)]
 \item The functor $H_{m|n}$ is full.
 \item For $X,Y\in\bD$ with $\deg X+\deg Y< 2(m+1)(n+1)$, $H_{m|n}$ induces an isomorphism
 $$\bD(X,Y)\;\stackrel{\sim}{\to}\Hom_{\GL(m|n)}(H_{m|n}(X),H_{m|n}(Y)).$$
 \item There exists an indecomposable object $Q$ in $\bD$ with $\deg Q=mn$, such that $H_{m|n}(Q)$ is projective in $ \Rep_k\GL(m|n)$.
 \end{enumerate}
 \end{lemma}
 \begin{proof}
 These statements are well-known, see e.g. \cite{He, Sergeev}. The precise statements can also be found in \cite[Theorem~7.2.1(ii)]{Selecta}, the paragraph above \cite[Corollary~7.2.2]{Selecta}, and \cite[Proposition~8.2.3(i)]{Selecta}.
 \end{proof}

\subsubsection{} A rigid object $X$ in a symmetric monoidal category is {\bf symmetrically self-dual} if $X^\vee\simeq X$ and for $\ev_X:X\otimes X\to\unit$, we have $\ev_X=\ev_X\circ\sigma_{X,X}$.
Following \cite[\S 9]{Deligne}, we have the category $[\OO_\normalt,R]_0$, which is the free $R$-linear symmetric monoidal category on one symmetrically self-dual object $U_\normalt$ of dimension $\normalt$. 
Its objects are tensor powers of $U_\normalt$.

The pseudo-abelian envelope $[\OO_\normalt,R]$ is thus a pseudo-tensor category over $R$. By construction, every object $X$ in $[\OO_\normalt,R]$ is a direct summand of a direct sum of objects $\otimes^i U_\normalt$. We denote by $\deg X$ the minimal $d\in\mN$ such that $X$ is a direct summand of a direct sum of $\otimes^i U_\normalt$ with $i\le d$. By \cite[Th\'eor\`eme~9.7]{Deligne}, $[\OO_\normalt,k]$ is a semisimple tensor category when $\charr(k)=0$ and $\normalt\not\in\mZ$.

\begin{lemma}[Proposition~9.4 \cite{Deligne}]\label{LemUniO}
Consider a pseudo-tensor category $\bD$ over $R$. Evaluation at $U_\normalt$ yields an equivalence between $\Tens([\OO_\normalt,R],\bD)$
and the groupoid of symmetrically self-dual objects of dimension $\normalt$ in $\bD$.
\end{lemma}

\subsubsection{}\label{FunOSp}  Set $\bD:=[\OO_\normalt,k]$ for an algebraically closed field $k$ of characteristic zero, for $\normalt\in\mZ\subset k$. Consider a non-degenerate (super)symmetric bilinear form on $k^{m|2n}\in\svecc$ and let $\OSp(m|2n)$ be the closed subgroup of $\GL(m|2n)$ which preserves the form.
As an application of Lemma~\ref{LemUniO}, there exists a tensor functor
$$F_{m|2n}:\bD\to \Rep_k\OSp(m|2n),\quad U_t\mapsto k^{m|2n}$$
 for every $m,n\in\mN$ with $m-2n=\normalt$.
 
 \begin{lemma}\label{FunOSpLem}
Retain the notation of \ref{FunOSp}.
 \begin{enumerate}[label=(\roman*)]
 \item The functor $F_{m|2n}$ is full.
 \item For $X,Y\in\bD$ with $\deg X+\deg Y< 2(m+1)(n+1)$, $F_{m|2n}$ induces an isomorphism
 $$\bD(X,Y)\;\stackrel{\sim}{\to}\;\Hom_{\OSp(m|2n)}(F_{m|2n}(X),F_{m|2n}(Y)).$$
 \item There exists an indecomposable object $Q$ in $\bD$ with $\deg Q=mn$, such that $F_{m|2n}(Q)$ is projective in $ \Rep_k\OSp(m|2n)$.
 \end{enumerate}
 \end{lemma}
\begin{proof}
Claim (i) is \cite[Theorem~5.3]{LZ-FFT}. Claim (ii) is \cite[7.1.1(ii) and 8.1.3(i)]{Selecta} or follows from \cite[Theorem~5.12]{Yang}. If $m\le 1$ or $n=0$, then $\Rep\OSp(m|2n)$ is semisimple, so claim (iii) becomes trivial.
The case $m>1$ and $n>0$ follows from the observation in \cite{CH} that the objects in $\bD$ sent to projective objects under $F_{m|2n}$ are the same ones which are sent to zero by $F_{m-2|2n-2}$, and the description of that kernel as in~\cite[Theorem~7.1.1]{Selecta}.
\end{proof}

%\subsubsection{}
%We have the pseudo-tensor category $[\SG_t,R]$ from \cite[Definition 2.17(ii)]{Deligne}. It is the pseudo-abelian envelope of a category which has as objects all finite sets. For a set $U$, we denote by $[U]$ the corresponding object in $[\SG_t,R]$ and we have $[U]\otimes [V]=[U\sqcup V]$. We also denote by $\ast$ a singleton. %Then $[\ast]$ has the structure of a commutative algebra in $[\SG_t,R]$. 
%By \cite[\S 8]{Deligne}, the object $[\ast]$ (which satisfies $\dim[\ast]=t$) has a certain algebra structure and $[\SG_t,R]$ has a universal property with respect to such algebras. If $\charr(k)=0$ and $t\not\in\mN$, then $[\SG_t,k]$ is a semisimple tensor %category, by \cite[Th\'eor\`eme~2.18]{Deligne}.

\section{Monoidal splitting and faithfulness}\label{SecSplit}
We fix a field $k$ and a pseudo-tensor category $(\bD,\otimes,\unit,\sigma)$ over $k$.

\subsection{Splitting of morphisms}

\begin{definition}
An object $X\in\bD$ {\bf splits a morphism} $f:A\to B$ in $\bD$ if $X\otimes f$ is split.
The category $\bD$ is {\bf self-splitting} if for every morphism $h$ in $\bD$ there exists an object which splits $h$.
\end{definition}

For an object $X\in\bD$, we will encounter the morphism
$$\cE_X:=\ev_X\otimes X^\vee\otimes X- X^\vee\otimes X\otimes \ev_X:\; \,X^\vee\otimes X\otimes X^\vee\otimes X\to X^\vee\otimes X$$
several times, hence we give it a name.

\begin{lemma}\label{LemSplit}
\begin{enumerate}[label=(\roman*)]
%\item An object $X$ splits a morphism $f$ if and only if every direct summand of $X$ splits $f$.
\item The morphisms $\ev_X$ and $\co_X$ are split by $X$ and by $X^\vee$.
\item The morphism $\cE_X$ is split by $X\otimes X^\vee$.

%\item Consider morphisms $A\stackrel{f_1}{\to}B\stackrel{f_2}{\to}C$
%\item Consider objects $A,B,C$ in $\bD$ and the isomorphism
%$$\phi:\bD(A\otimes B,C)\stackrel{\sim}{\to}\bD(A,B^\ast\otimes C).$$
%If $X\in\bD$ splits $f:A\otimes B\to C$, then $X\otimes B$ splits $\phi(f)$. If $Y\in\bD$ splits $g:A\to B^\ast\otimes C$, then $Y\otimes B^\ast$ splits $\phi^{-1}(g)$.
%\item Consider morphisms $f_i:A_i\to B_i$ and objects $X_i$ such that $X_i$ splits $f_i$, for $i\in\{1,2\}$, then $X_1\otimes X_2$ splits $f_1\otimes f_2$.

\end{enumerate}
\end{lemma}
\begin{proof}
It follows from \eqref{snake} that $f:=X\otimes \ev_X$ is split, with $g:=\co_X\otimes X$, which proves part (i).
It follows similarly that $f:=X\otimes \cE_X\otimes X^\vee$ is split, with 
$$g:=X\otimes X^\vee\otimes X\otimes X^\vee\otimes \co_X-\co_X\otimes X\otimes \ev_X\otimes X^\vee\otimes \co_X,$$ which proves part (ii). 
\end{proof}

%\begin{remark}
%By Lemma~\ref{LemUni}, we can interpret (the proof of) Lemma~\ref{LemSplit} as follows. In $[\GL_\normalt,k]$, the morphisms $\ev_{V_\normalt}\otimes V_\normalt$ and $\cE_{V_\normalt}\otimes V_\normalt\otimes V^\vee_\normalt$ are split. The same is therefore true for arbitrary $X\in\bD$, where the morphisms required for the splitting are in the image of $[\GL_\normalt,k]\to\bD, V_\normalt\mapsto X$, with $\normalt:=\dim X$.
%\end{remark}

The following lemma is well-known.
\begin{lemma}\label{LemP}
Assume $\bD$ is a tensor category and take $X\in\bD$. The following are equivalent:
\begin{enumerate}[label=(\roman*)]
\item $X$ is projective
\item $X$ is injective.
\item $X\otimes f$ is split for every morphism $f$ in $\bD$.
\end{enumerate}
\end{lemma}
\begin{proof}
First we show that (i) implies (iii). For a morphism $f:M\to N$ we denote the image and cokernel by $A$ and $B$. By adjunction, $X\otimes D$ is projective, for every $D\in\bD$.  Consequently $X\otimes M\tto X\otimes A$ and $X\otimes N\tto X\otimes B$ split, from which it follows that $X\otimes f$ is split. That (ii) implies (iii) is proved similarly.

Now if (iii) is satisfied, then it follows by adjunction that $X^\vee$ is both projective and injective. Also by adjunction, the fact that $X^\vee$ is projective (resp. injective) implies that $X$ is injective (resp. projective). Hence (iii) implies (i) and (ii).
\end{proof}

\subsection{Faithfulness of objects}

\begin{definition}\label{DefF}
An object $X\in\bD$ is {\bf faithful} if one of the following two equivalent conditions is satisfied:
\begin{enumerate}[label=(\roman*)]
\item The functor $X\otimes-:\bD\to\bD$ is faithful.
\item The evaluation $\ev_X:X^\vee\otimes X\to \unit$ is an epimorphism in $\bD$.
\end{enumerate}
\end{definition}
For our applications, we will need a strictly stronger notion than the above faithfulness.
\begin{definition}\label{DefSF}
An object $X\in\bD$ is {\bf strongly faithful} if one of the following two equivalent conditions is satisfied:
\begin{enumerate}[label=(\roman*)]
\item For every $M,N\in\bD$, the sequence
$$0\to\bD(M,N)\xrightarrow{X\otimes -}\bD(X M,X N)\xrightarrow{(X\otimes-)-(s\otimes N)(X\otimes -)(s\otimes M)}\bD(X X M,X X N),$$
with $s=\sigma_{XX} $, is exact in $\Vecc$.
\item The sequence
$$\gamma_X:\;\,X^\vee\otimes X\otimes X^\vee\otimes X\xrightarrow{\cE_X}X^\vee\otimes X\xrightarrow{\ev_X}\unit\to 0$$
is exact in $\bD$, meaning $\gamma_X\in\Xi(\bD)$.
\end{enumerate}
\end{definition}
Clearly $X$ is (strongly) faithful if and only if $X^\vee$ is (strongly) faithful.
Examples of faithful objects which are {\bf not} strongly faithful will be given in the next subsection.

\begin{example}\label{ExampUni} ${}$
\begin{enumerate}[label=(\roman*)]
\item The unit $\unit$ is strongly faithful in any pseudo-tensor category $\bD$.
\item The objects $V_\normalt$ and $U_\normalt$ in $[\GL_\normalt,k]$ and $[\OO_\normalt,k]$ are strongly faithful. This follows easily from the diagrammatic calculus and version (i) of Definition~\ref{DefSF}.
\end{enumerate}
\end{example}
We say that $X\in\bD$ {\bf reflects cokernels} when every sequence $\gamma$ as in \eqref{rexeq} is exact if and only if $X\otimes \gamma$ is exact.  
Note that one direction of the condition is automatic by Remark~\ref{RemSigma}. 
Reflecting kernels is defined similarly. 
Remark~\ref{RemSigma} also shows that $X\otimes Y$ reflects cokernels if and only if  both $X$ and $Y$ reflect cokernels; a fact that we will use freely.

%\begin{definition}
%An object $Y\in\bD$ {\bf reflects exactness} if one of the following two equivalent conditions is satisfied:
%\begin{enumerate}[label=(\roman*)]
%\item The functor $Y\otimes-:\bD\to\bD$ reflects kernels and cokernels in $\bD$.
%\item If for a sequence $\gamma: X_2\to X_1\to X_0\to 0$ in $\bD$, one of the sequences $\{Y\otimes\gamma,Y^\vee\otimes\gamma\}$ are exact, {\it i.e.} is in $\Xi(\bD)$, then $\gamma\in\Xi(\bD)$.
%\end{enumerate}
%\end{definition}

%The above definition is motivated by the following observation.

\begin{lemma}\label{gammaXX}
For any $X\in\bD$, the sequence $X^\vee\otimes X\otimes \gamma_X$ is split exact.
\end{lemma}
\begin{proof}
By Lemma~\ref{LemSplit}, the sequence $\xi_X:=X^\vee\otimes X\otimes \gamma_X$  is split. Set $\normalt=\dim X$. Then $\xi_X$ is the image of $\xi_{V_\normalt}$ under the tensor functor $[\GL_\normalt,k]\to \bD$ corresponding to $V_\normalt\mapsto X$ in Lemma~\ref{LemUni}. Since $\xi_{V_{\normalt}}$ is split exact, by Example~\ref{ExampUni}(ii), and tensor functors are additive, also $\xi_X$ is split exact. 
\end{proof}

\begin{prop}\label{Refl} The following are equivalent for $X\in\bD$.
\begin{enumerate}[label=(\roman*)]
\item $X$ is strongly faithful.
\item $X\otimes X^\vee$ reflects cokernels.
\item $X$ reflects both kernels and cokernels.
\end{enumerate} 
\end{prop}
\begin{proof}
Assume first that $X$ is strongly faithful and consider a sequence $X_2\to X_1\to X_0$ in $\bD$. Tensoring with $\gamma_X$ yields a commutative diagram
$$\xymatrix{
0&0&0&\\
X_2\ar[r]\ar[u]&X_1\ar[r]\ar[u]& X_0\ar[r]\ar[u]&0\\
X^\vee X X_2\ar[r]\ar[u]&X^\vee X X_1\ar[r]\ar[u]&X^\vee X X_0\ar[r]\ar[u]&0\\
X^\vee XX^\vee X X_2\ar[r]\ar[u]&X^\vee XX^\vee X X_1\ar[r]\ar[u]&X^\vee XX^\vee X X_0\ar[r]\ar[u]&0
}$$
with exact columns. If the second row is exact, then so is the third. It then follows from elementary diagram chasing that the first row is also exact. Hence $X^\vee \otimes X$ reflects cokernels.

Now assume that $X^\vee\otimes X$ reflects cokernels. By Lemma~\ref{gammaXX}, application of the functor $X^\vee\otimes X\otimes -$ to the sequence $\gamma_X$ yields an exact sequence. Hence also $\gamma_X$ is exact and $X$
 is strongly faithful by definition. This already shows that (i) and (ii) are equivalent.
 
Claim (ii) is equivalent to the claim that both $X$ and $X^\vee$ reflect cokernels. By adjunction, $X^\vee$ reflects cokernels if and only if $X$ reflects kernels. Hence (ii) and (iii) are equivalent.
\end{proof}

\begin{prop}\label{PropSF}
Let $X,Y$ be objects in $\bD$.
\begin{enumerate}[label=(\roman*)]
\item $X$ and $Y$ are strongly faithful if and only if $X\otimes Y$ is strongly faithful.
\item If $\dim X\not=0$, then $X$ is strongly faithful.
\item If $\bD$ is a tensor category and $X\not=0$, then $X$ is strongly faithful.
\end{enumerate}
\end{prop}
\begin{proof}
Part (i) follows from Proposition~\ref{Refl}.

If $d:=\dim X$ is invertible, then consider the morphisms
$$f:=\frac{1}{d}\sigma_{X,X^\vee}\circ\co_X:\unit\to X^\vee\otimes X$$
and
$$((f\otimes f)\circ \ev_X - X^\vee\otimes X\otimes f):X^\vee\otimes X\to X^\vee\otimes X\otimes X^\vee\otimes X.$$
It follows from direct computation that these ensure the sequence in Definition~\ref{DefSF}(ii) is split exact. This proves part (ii).

Part (iii) follows from Proposition~\ref{Refl}, since all non-zero objects in tensor categories reflect cokernels.
\end{proof}
We can also prove \ref{PropSF}(ii) directly from Definition~\ref{DefSF}, using a `monoidal analogue' of the diagram in \cite[IV.1.7]{SGA}.
The following corollary is a direct consequence of \ref{PropSF}(iii).
\begin{corollary}\label{CorSF}
If $\bD$ admits a fully faithful tensor functor into a tensor category, every non-zero object in $\bD$ is strongly faithful.
\end{corollary}

\begin{lemma}\label{LemFieldExt}
Consider a field extension $K/k$.
\begin{enumerate}[label=(\roman*)]
\item If $X\in\bD$ is strongly faithful in $\bD_K$, it is also strongly faithful in $\bD$.
\item If $f\in \bD(X,Y)$ interpreted in $\bD_K$ is split, then $f$ is also split in $\bD$.
\end{enumerate}
\end{lemma}
\begin{proof}
Part (i) follows from applying either version of Definition~\ref{DefSF} and using the fact that the functor $K\otimes_k-$ from $\Vecc_k$ to $\Vecc_K$ is faithful and exact.

%Under the assumptions in part (i), it is clear that $X$ is faithful in $\bD$. Next consider $g\in\bD(X^\vee X,Y)$ with $g\circ\cE_X=0$. By assumption, there exists 
%$$h\,\in\,K\otimes_k\bD(\unit, Y)=\bD(\unit,Y)\oplus \left(V\otimes_k\bD(\unit, Y)\right)$$
%such that $g=h\circ \ev_X$. According to the above decomposition we have some finite sum $h=h_0+\sum_{i>1}v_i\otimes h_i$, with $h_i\in\bD(\unit,Y)$ and $v_i\in V$. Consequently, 
%we have
%$$g=h_0\circ \ev_X+\sum_{i>1}v_i\otimes (h_i\circ \ev_X) \;\in \,\bD(X^\vee X,Y)\oplus \left(V\otimes_k\bD(X^\vee X,Y)\right).$$
%Hence, the second term on the right-hand side is zero, which concludes the proof of part (i).

For part (ii), by assumption, we have $g\in K\otimes_k\bD(Y,X)$ with $f\circ g\circ f=f$. We fix a complement $V$ in $K$ of the canonical $k$-subspace $k\subset K$. We have $g=g_0+g_1$ with $g_0\in\bD(Y,X)$ and $g_1\in V\otimes_k\bD(Y,X)$. It follows immediately that $f\circ g_0\circ f=f$.
\end{proof}

\begin{lemma}\label{LemEO}
Consider $X,Y,Z\in\bD$ such that $Y$ is a direct summand of $ X\otimes Z$. Then the sequence $\bD(\gamma_X,Y)$ is exact in $\Vecc$.
\end{lemma}
\begin{proof}
Since $X$ is a direct summand of $X\otimes X^\vee\otimes X$, it follows that $Y$ is also a direct summand of $X^\vee\otimes X\otimes Z'$, for $Z':=X\otimes Z$.
By functoriality and adjunction, it therefore suffices to prove that
$$\bD(X^\vee\otimes X\otimes \gamma_X,Z')$$
is exact. The latter is a consequence of Lemma~\ref{gammaXX}.
\end{proof}

\subsection{Examples}

\begin{theorem}\label{ThmBrauer}
If $\charr(k)=0$ and $\normalt\in \mZ$, the categories $[\GL_\normalt,k]$ and $[\OO_\normalt,k]$ are self-splitting and every non-zero object is strongly faithful.
\end{theorem}
\begin{proof}
By Lemma~\ref{LemFieldExt} and the fact that $[\GL_\normalt,k]_K\simeq [\GL_\normalt,K]$ and $[\OO_\normalt,k]_K\simeq [\OO_\normalt,K]$ for any field extension $K/k$, it is sufficient to prove the theorem for algebraically closed $k$. 

Set $\bD:=[\OO_\normalt,k]$. We start by proving the claim about strong faithfulness.
 It follows immediately from Definition~\ref{DefF}(ii) and the diagrammatic calculus that all objects in $\bD$ are faithful. For $0\not=X\in\bD$ we need to demonstrate that for a given morphism $f:X^\vee\otimes X\to Y$ in $\bD$ with $f\circ\cE_X=0$, there exists $g:\unit\to Y$ such that $f=g\circ \ev_X$. Set $a=\deg X$ and $b=\deg Y$. Take $m,n\in\mN$ with $m-2n=\normalt$ and $2a+b\le 2(m+1)(n+1)$ and consider the tensor functor
$$F_{m|2n}:\bD\to\Rep_{k}\OSp(m|2n)$$
from \ref{FunOSp}. By Proposition~\ref{PropSF}(iii) and the fulness of $F_{m|2n}$ in Lemma~\ref{FunOSpLem}(i), there exists a morphism $g:\unit\to Y$ in $\bD$ such that $F_{m|2n}(f)=F_{m|2n}(g\circ \ev_X)$. By Lemma~\ref{FunOSpLem}(ii), this implies that $f=g\circ\ev_X$, as desired.

Now consider an arbitrary morphism $f:A\to B$ in $\bD$ and set $a=\deg A$ and $b=\deg B$. Take $m,n\in\mN$ with $m-2n\in\normalt$ and $a+b\le m+n$. The latter inequality implies
\begin{equation}\label{eqabmn}a+b+2mn < 2(m+1)(n+1).\end{equation}
We consider again the functor $F_{m|2n}$. By Lemma~\ref{FunOSpLem}(iii), there exists $Q$ in $\bD$, with $\deg Q=mn$, such that $F_{m|2n}(Q)$ is projective. By Lemma~\ref{LemP}, 
$$f':=F_{m|2n}(Q\otimes f)\; :\; F_{m|2n}(Q\otimes A)\to F_{m|2n}(Q\otimes B)$$ is split. Thus there exists a morphism $g'$ in $\Rep\OSp(m|2n)$ such that $f'g'f'=f'$. By Lemma~\ref{FunOSpLem}(ii) and the inequality \eqref{eqabmn}, there thus exists a morphism $Q\otimes B\to Q\otimes A$ which ensures that $Q\otimes f$ is split. Hence $\bD$ is self-splitting.

 The claims for $[\GL_\normalt,k]$ are similarly proved using Lemma~\ref{FunGLLem} 
\end{proof}

\begin{remark}
That non-zero objects in $[\GL_\normalt,k]$ are strongly faithful when char$(k)=0$, also follows from Corollary~\ref{CorSF} and Deligne's result in \cite[Proposition~10.17]{Deligne}, which states that $[\GL_\normalt,k]$ admits a fully faithful tensor functor into a tensor category.
\end{remark}

\subsubsection{} For a commutative $k$-algebra $K$, consider the subcategory $\bC$ of $[\GL_0,K]_0$
which has the same objects and for which the inclusion functor $\bC\to[\GL_0,K]_0$ is full on each morphism set, except that on $\bC(\unit,\unit)$ it realises the unit morphism $k\to K$. That $\bC$ constitutes a (monoidal) subcategory of $[\GL_0,K]_0$ follows from the fact that the collection of all morphisms in $[\GL_0,K]_0$ excluding the ones $\unit\to\unit$ form a (tensor) ideal. 

Now the pseudo-abelian envelope $\bD$ of $\bC$ is a pseudo-tensor category over $k$.

\begin{lemma}\label{FunnyEx}
\begin{enumerate}[label=(\roman*)]
\item The object $V_0$ in $\bD$ is faithful but not strongly faithful, unless $K=k$.
\item If $\charr(k)=0$ and $K/k$ is a field extension then $\bD$ is self-splitting and every non-zero object is faithful.
\end{enumerate}
\end{lemma}
\begin{proof}
Part (ii) can be derived from Theorem~\ref{ThmBrauer}.

 For part (i), we consider the sequence in \ref{DefSF}(i) for $M=N=\unit$ and $X=V_0$, which yields
$$0\to k\to K\to K \mathrm{S}_2,$$
where $\mathrm{S}_2$ is the symmetric group on two symbols, the morphism $k\to K$ is the unit morphism and the morphism $K\to K\mathrm{S}_2$ is zero. This follows either by direct computation or from the fact that $V_0$ is strongly faithful in $[\GL_0,K]$ which shows that the kernel of the morphism $K\to K \mathrm{S}_2$ is $K$, the endomorphism ring of $\unit$ in $[\GL_0,K]$.
\end{proof}

\begin{question}
For a strongly faithful $X\in\bD$, by definition $\ev_X$ is a normal epimorphism (a cokernel). In the example in Lemma~\ref{FunnyEx}, the epimorphism $\ev_{V_0}$ is not even strict, so certainly not normal. Are there examples of pseudo-tensor categories with objects $X$ for which $\ev_X$ is a normal epimorphism while $X$ is not strongly faithful?
\end{question}

%%%%%%%%%%%%%%%%%%%%%%%%%%%%%%%%%%%%%%%%%%%%%%%%%%%%%%%%%%%%%%%%%%%%%%%%%%%%%%%%%%%%%%%%%%

\section{A closed monoidal Grothendieck category}\label{SecSh}

Fix an arbitrary pseudo-tensor category $\bD$ over a field $k$. 

\subsection{The category of sheaves}
\subsubsection{}
We consider the $k$-linear presheaf category $\PSh\bD$ of $k$-linear functors $\bD^{\op}\to\Vecc_k$. 
Then $\PSh\bD$ is symmetric closed monoidal for the Day convolution $\star$, see e.g.~\cite[\S 3.2]{Sch2}. The tensor product of two presheaves $F,G$ is given by the co-end expression
$$F\star G\;:=\;\int^{X,Y\in\bD}F(X)\otimes_k G(Y)\otimes_k \bD(-,X\otimes Y),$$
and the internal Hom is given by
$$[G,H]\;=\; \int_{Y\in \bD}\Hom_k(G(Y),H(-\otimes Y)).$$
The Yoneda embedding $Y:\bD\to\PSh\bD$ is canonically symmetric monoidal.  By construction, the tensor product $-\star-$ is cocontinuous in each variable.

% which is equivalent here (by Freyd's special adjoint functor theorem) to the fact that the category is monoidally closed. 

\subsubsection{}\label{DefSigma} 
We define the full subcategory $\Sh\bD$ of $F\in\PSh\bD$ for which $F(D\otimes\gamma_X)$ is exact in $\Vecc$, for every exact sequence
\begin{equation}
\label{eqDX}
D\otimes\gamma_X:\;\;DX^\vee XX^\vee X\to DX^\vee X\to D\to 0,
\end{equation}
with $X$ strongly faithful and $D$ arbitrary in $\bD$. Since `limits commute', it follows that the inclusion functor from $\Sh\bD$ to $\PSh\bD$ is continuous and hence (by Freyd's special adjoint functor theorem) admits a left adjoint 
\begin{equation}\label{reflection}S\,:\,\PSh\bD\to\Sh\bD,\end{equation} the sheafification or reflection. The restriction of $S$ to $\Sh\bD$ is the identity. If $F\in \Sh\bD$, then clearly the functor $F(-\otimes Z)$ is also in $\Sh\bD$, for each $Z\in \bD$. It then follows as a direct application of Day's reflection theorem~\cite[Theorem~1.2(2)]{Day} that there is a unique closed symmetric monoidal structure on $\Sh\bD$ which makes $S$ symmetric monoidal. We denote the tensor product on $\Sh\bD$ again by $\otimes$, and by definition we have
$$F\otimes G\;:=\; S(F\star G),\qquad \mbox{for $F,G\in\Sh\bD$}.$$
The Yoneda embedding $Y:\bD\to\PSh\bD$ factors through the embedding of the subcategory $\Sh\bD$. We will denote the corresponding fully faithful functor by $Y_0:\bD\to\Sh\bD$. It is isomorphic to the composite $S\circ Y$, so in particular $Y_0$ is symmetric monoidal. 

\subsubsection{}\label{DefTop}
We refer to Appendix~\ref{AppTop} for the notions of sieve, Grothendieck topology, the category of sheaves with respect to a topology and localisations of Grothendieck categories. 

For each $D\in\bD$, denote by $\cT(D)$ the set of all sieves $R\subset \bD(-,D)$ such that there exists a strongly faithful $X\in\bD$ for which $D\otimes \ev_X\in R(DX^\vee X ) $.
Our notation $\Sh\bD$ is justified by the following theorem.
\begin{theorem}\label{ThmGro}
\begin{enumerate}[label=(\roman*)]
\item The assignment $D\mapsto \cT(D)$ from \ref{DefTop} is a $k$-linear Grothendieck topology on $\bD$ and the subcategory $\Sh\bD$ of $\PSh\bD$ is precisely the category of $\cT$-sheaves $\Sh(\bD,\cT)$.
\item $\Sh\bD$ is a localisation of $\PSh\bD$, so in particular a Grothendieck category.
\item Every object in $\bD$ is compact in $\Sh\bD$ and every object in $\Sh\bD$ is a quotient of a (possibly infinite) coproduct of objects in $\bD$.
\item Given a functor $F:\mathsf{J}\to \Sh \bD$ out of a filtered category $\mathsf{J}$, its colimit taken in $\PSh\bD$ is contained in $\Sh\bD$ (and hence equal to the colimit of $F$ in there).
\item The functor $Y_0:\bD\to\Sh\bD$ sends $D\otimes \gamma_X$ to an exact sequence in $\Sh\bD$, for every $D\in\bD$ and strongly faithful $X\in\bD$.
\end{enumerate}

\end{theorem}
\begin{proof}
Part (i) will be proved in Subsection~\ref{SecPf}. We explain how (i) implies (ii) in the usual fashion.
As a left adjoint, the reflection $S$ in \eqref{reflection} is cocontinuous. This already implies that $\Sh\bD$ is a cocomplete. It also follows that the coproduct in $\Sh\bD$ 
$$G:=\bigoplus_{X\in\Ob\bD/\simeq} X$$
 over the set of isomorphism classes of objects in $\bD$, is a generator of $\Sh\bD$, meaning that $\Sh\bD(G,-):\Sh\bD\to\Vecc$ is faithful.
Hence it suffices to show that $\Sh\bD$ is abelian and that direct limits of short exact sequences are (left) exact. Both properties follow easily if $S:\PSh\bD\to\Sh\bD$ is (left) exact, {\em i.e.} when $\Sh\bD$ is a localisation of $\PSh\bD$. Hence claim (ii) follows from claim (i) and Theorem~\ref{ThmBQ}.

For part (v), we can observe that by definition and the Yoneda lemma
$$\Sh\bD(Y_0(D\otimes\gamma_X),F)\;=\;\PSh\bD(Y(D\otimes\gamma_X),F)\;=\;F(D\otimes\gamma_X),$$
for arbitrary $F\in\Sh\bD$. Hence $Y_0(D\otimes\gamma_X)$ is indeed exact.

Part (iv) follows easily from the fact that in $\Vecc$, a filtered colimit of short exact sequences is exact. 

Finally, we prove part (iii). Since by construction objects in $\bD$ are compact in $\PSh\bD$, part (iv) implies that $Y_0:\bD\to \Sh\bD$ sends every object in $\bD$ to a compact object in $\Sh\bD$. That every object is a quotient of coproduct of objects in $\bD$ follows from the above fact that $G$ is a generator.
\end{proof}

\subsection{Proof of \ref{ThmGro}(i) }\label{SecPf}
Here we complete the proof of Theorem~\ref{ThmGro}. As heuristic explanation of \ref{ThmGro}(i) we also present a non-enriched but similar site in Analogy~\ref{Anal}. 

\subsubsection{}
First we prove that $\cT$ from \ref{DefTop} constitutes a topology as in Definition~\ref{Def1}. Condition (T1) is immediate from Example~\ref{ExampUni}(i). For condition (T2), consider $A\in\bD$, $R\in \cT(A)$ and a morphism $f:B\to A$ in $\bD$. By definition, there exists a strongly faithful $X\in\bD$ such that $A\otimes\ev_X$ is in $R$. It then follows that $B\otimes\ev_X$ is in $f^{-1}R$, so $f^{-1}R\in\cT(B)$.

For Condition (T3) consider $S\subset\bD(-,A)$ and $R\in \cT(A)$ as in (T3). Since there exists $f:=A\otimes\ev_X$ in $R(AX^\vee X)$, for some strongly faithful $X$, there must exist
 a strongly faithful $Y\in\bD$ such that $AX^\vee X\ev_Y$ is in $f^{-1}S(AX^\vee X Y^\vee Y)$. The latter just means that $A\otimes\ev_X\otimes\ev_Y$ is in $S(AX^\vee XY^\vee Y)$, which means that also $A\otimes\ev_{X\otimes Y}$ is in $S(AY^\vee X^\vee XY)$. It then follows from Proposition~\ref{PropSF}(i) that $S\in\cT(A)$.
 
% Hence $\cT$ is indeed a $k$-linear Grothendieck topology. 

 \subsubsection{}\label{Pf2} Now we prove the equality $\Sh\bD=\Sh(\bD,\cT)$. Take an arbitrary presheaf $F\in\PSh\bD$. By a `pair' $(D,X)$ we mean an arbitrary $D\in\bD$ and a strongly faithful $X\in\bD$.
 For each pair $(D,X)$, denote by $R^D_X$ the sieve on $D$ generated by the morphism $D\otimes \ev_X$. This is the minimal sieve on $D$ containing $D\otimes \ev_X$, or equivalently the image of 
 $$\bD(-,DX^\vee X)\xrightarrow{(D \otimes\ev_X)\circ-} \bD(-,D).$$
 Since every $R\in \cT(D)$ is of the form $R^D_X\subset R\subset \bD(-,D)$ for some strongly faithful $X$, it follows from Definition~\ref{Def2} that $F\in\PSh\bD$ is a $\cT$-sheaf if and only if
 $$F(D)\;\to\; \Nat(R^D_X,F)$$
is an isomorphism for every pair $(D,X)$. Since the representable objects in $\bD$ yield a set of generators for $\PSh\bD$, we can complete the epimorphism $\bD(-, DX^\vee X)\tto R^D_X$ to an exact sequence
$$\bigoplus_{g:B\to DX^\vee X, f\circ g=0}\bD(-,B)\;\to\; \bD(-, DX^\vee X)\;\to\; R^D_X\;\to\; 0,$$
in $\PSh\bD$ with $f:=D\otimes\ev_X$. In other words, $F$ is a $\cT$-sheaf if and only if the sequence
\begin{equation}\label{eqF1}0\to F(D)\xrightarrow{F(f)} F(DX^\vee X)\to  \prod_{g:B\to DX^\vee X, f\circ g=0}F(B)\end{equation}
is exact for every pair $(D,X)$.
On the other hand, by definition, $F\in\Sh\bD$ if and only if 
\begin{equation}\label{eqF2}0\to F(D)\xrightarrow{F(f)}  F(DX^\vee X)\xrightarrow{F(D\cE_X)}F(DX^\vee XX^\vee X)\end{equation}
is exact for every pair $(D,X)$.

For any pair $(D,X)$, clearly the sequence \eqref{eqF1} is exact whenever \eqref{eqF2} is exact. On the other hand, assume that \eqref{eqF1} is exact, for a fixed $X$ but for every $D\in\bD$. For any $g:B\to DX^\vee X$ with $f\circ g=0$, we have the following commutative diagram
$$\xymatrix{
DX^\vee XX^\vee X\ar[rr]^-{D\otimes\cE_X}&&DX^\vee X\\
BX^\vee X\ar[rr]_-{B\otimes \ev_X}\ar[u]^{-g\otimes X^\vee X}&&B\ar[u]_g.
}$$ 
Applying $F$ yields a commutative diagram
$$\xymatrix{
F(DX^\vee XX^\vee X)\ar[d]&&F(DX^\vee X)\ar[ll]_-{F(D\cE_X)}\ar[d]^{F(g)}\\
F(BX^\vee X)&&F(B),\ar@{_{(}->}[ll]
}$$ 
and the fact that the lower horizontal arrow is a monomorphism follows from our assumption that \eqref{eqF1} with $D$ replaced by $B$ be exact.
Consider $a\in F(DX^\vee X)$ such that $F(D\cE_X)(a)=0$. By commutativity of the diagram (and using the monomorphism) we find that also $F(g)(a)=0$. Exactness of \eqref{eqF1} thus implies that $a$ is the image of $F(f)$ and therefore \eqref{eqF2} is exact too. This concludes the proof of the claim $\Sh\bD=\Sh(\bD,\cT)$.

\begin{analogy}\label{Anal}
Consider a category $\bB$ with finite products, with terminal object~$\ast$. By \cite[IV.1.3]{SGA}, the morphism $U\to\ast$ is a `universal effective epimorphism' if the induced $V\times U\to V $ is an effective epimorphism for every $V\in\bB$.
By \cite[IV.1.8]{SGA}, if $U\to\ast$ and $U'\to\ast$ are universal effective epimorphisms, the same is true for $U\times U'$.
 Take $C\subset\Ob\bB$, containing $\ast$ and closed under products, such that $U\to \ast$ is a universal effective epimorphism for every $U\in C$. The corresponding collection of coverings $V\times U\to V$ forms a classical (non-enriched) Grothendieck (pre)topology. The sheaves are the presheaves $F:\bB^{\op}\to\Set$ for which
$$F(V)\to F(V\times U)\rightrightarrows F(V\times U\times U)$$
is an equaliser for each $V\in\bB$ and $U\in C$. In particular the representable presheaves are sheaves.
\end{analogy}

\begin{remark}
Denote by $\Sigma\subset\Xi(\bD)$ the class of all exact sequences 
$$D\otimes\gamma_X\quad\mbox{and}\quad D\otimes (X^\vee\otimes X\xrightarrow{\ev_X} \unit\to 0\to 0),$$ for arbitrary $D\in \bD$ and strongly faithful $X\in\bD$. 
It follows easily, and from similar arguments as used in \ref{Pf2}, that $\Sigma$ constitutes an `ind-class', as in Definition \ref{DefInd} in the appendix. We can therefore also prove Theorem~\ref{ThmGro}(i) by applying Propositions~\ref{LemSch} and~\ref{LemSch2}.
\end{remark}
\subsection{Connection with abelian envelopes}

\begin{theorem}\label{ThmShEnv}
If $\Sh\bD$ is tensor equivalent to the ind-completion of a tensor category $\bT$ over $k$, then $\bT$ is the abelian envelope of $\bD$. 
\end{theorem}
\begin{proof}
If $\Sh\bD$ is equivalent to the ind-completion of a tensor category, we can define a tensor category $\bT$ as the full subcategory of $\Sh\bD$ of all rigid objects, by Lemma~\ref{LemInd}. Since $\bD$ is rigid, this means that $Y_0$ takes values in the subcategory $\bT$. This allows us to construct a fully faithful tensor functor $F:\bD\to\bT$, which admits a commutative diagram
\begin{equation}\label{comm1}\xymatrix{
\bT\ar@{^{(}->}[r]&\Ind\bT\\
\bD\ar@{^{(}->}[u]^{F}\ar@{^{(}->}[r]^{Y_0}&\Sh\bD.\ar[u]^{\sim}
}\end{equation}

We introduce the category $\Tens^{rex}(\bT,-)$ of right exact tensor functors, the category $\Tens^{\gamma}(\bD,-)$ of tensor functors which send every exact sequence \eqref{eqDX}, for $X$ strongly faithful, to an exact sequence and the category $\Tens^{cc}$ of all cocontinuous tensor functors. For each tensor category $\bT_1$, diagram~\eqref{comm1} (and Theorem~\ref{ThmGro}(v)) induces a commutative diagram
\begin{equation}\label{comm2}\xymatrix{
\Tens^{rex}(\bT,\bT_1)\ar[r]\ar[d]& \Tens^{rex}(\bT,\Ind\bT_1)\ar[d]& \Tens^{cc}(\Ind\bT,\Ind\bT_1)\ar[l]\ar[d] \\\
\Tens^{\gamma}(\bD,\bT_1)\ar[r]&\Tens^{\gamma}(\bD,\Ind\bT_1)&\Tens^{cc}(\Sh\bD,\Ind\bT_1)\ar[l],
}\end{equation}
where each functor is given by composition with a tensor functor. The right vertical arrow is an equivalence, since it is induced from a tensor equivalence. Inverses of the two right horizontal arrows are given by taking left Kan extensions, see \cite[Theorem~3.2.4]{Sch2}.
Consequently, the middle vertical arrow is also an equivalence. The two left horizontal arrows are equivalences since any tensor functor from a pseudo-tensor category to the ind-completion of a tensor category takes values in rigid objects. We can thus use the equivalence between $\bT_1$ and the category of rigid objects in $\Ind\bT_1$ from Lemma~\ref{LemInd} to construct inverses. Consequently, also the left vertical arrow is an equivalence.

Now we will argue that the latter equivalence can be rewritten as the equivalence required by Definition~\ref{DefAbEnv}. Firstly, by Remark~\ref{RemEF}, we have
\begin{equation}\label{rexex}\Tens^{rex}(\bT,\bT_1)\,=\,\Tens^{ex}(\bT,\bT_1)\,\subset\,\Tens^{faith}(\bT,\bT_1).\end{equation}
 
 We  claim that we always have an inclusion
$$\Tens^{faith}(\bD,\bT_1)\,\subset\,\Tens^{\gamma}(\bD,\bT_1).$$
Indeed, a faithful tensor functor $H:\bD\to\bT_1$ maps every non-zero object in $\bD$ to a non-zero object in $\bT_1$. By Proposition~\ref{PropSF}(iii) every non-zero object in $\bT_1$ is strongly faithful, from which it follows that
$H$ sends every sequence \eqref{eqDX} to an exact sequence.

Moreover, by \eqref{rexex} and the left equivalence in \eqref{comm2}, every functor $H$ in $\Tens^{\gamma}(\bD,\bT_1)$ extends to a faithful functor $\bT\to\bT_1$, hence $H$ must be faithful as well. In particular $\Tens^{faith}(\bD,\bT_1)$ is equal to $\Tens^{\gamma}(\bD,\bT_1)$. Combining that equality with the equality in \eqref{rexex} and the equivalence on the left in diagram~\eqref{comm2} completes the proof.
\end{proof}

\begin{remark}\begin{enumerate}[label=(\roman*)]
\item Theorem~\ref{ThmShEnv} is not specific to $\Sh\bD$. Indeed, the same statement is true for instance for $\PSh\bD$ itself. However, if there exists a full subcategory $\bD\subset\bC\subset\PSh\bD$ which is equivalent to the ind-completion of a tensor category, then it follows from Proposition~\ref{PropSF}(iii) and the Yoneda lemma that $\bC\subset \Sh\bD$.
\item If $\bD$ is a semisimple tensor category, then $\Sh\bD=\PSh\bD=\Ind\bD$.
\end{enumerate}
\end{remark}

Motivated by Theorem~\ref{ThmShEnv}, we provide an explicit criterion for when $\Sh\bD$ is equivalent to the ind-completion of a tensor category.

\begin{prop}\label{PropShInd}
The following conditions are equivalent:
\begin{enumerate}[label=(\roman*)]
\item $\Sh\bD$ is tensor equivalent to the ind-completion of a tensor category over $k$.
\item There exists $M\in\Sh\bD$, with $M\otimes-:\Sh\bD\to\Sh\bD$ faithful and exact, which splits every morphism in $\bD$.
\item For every morphism $f$ in $\bD$, there exists $M\in\Sh\bD$, with $M\otimes-$ faithful and exact, for which $M\otimes f$ is split.
\end{enumerate}
\end{prop}
\begin{proof}
First we show that (i) implies (ii). Assume that $\Sh\bD\simeq\Ind\bT$ for a tensor category $\bT$. The category $\Ind\bT$ is a Grothendieck category and thus has enough injective objects. We take a non-zero $Y\in\bT$ and an injective object $I\in\Ind\bT$ which contains $Y$ as a subobject.  As for any object in $\Ind\bT$, the functor $I\otimes-$ is exact. Furthermore, if $I\otimes N=0$, for $N\in\Ind\bT$, then the subobject $Y\otimes N$ is also zero. However, $N$ is a subobject of $Y^\vee\otimes Y\otimes N$, which implies $N=0$. Hence $I\otimes -$ is faithful.
By applying adjunction, it follows that $X\otimes I$ is also injective for any rigid object $X$. Consider a morphism $f:X\to Y$ in $\bD\subset \bT$. Since $\bT$ is an abelian subcategory of $\Sh\bD\simeq\Ind\bT$, the image and kernel of $f$, which we denote by $Z$ and $K$, are in $\bT$ and hence also rigid. Then clearly $I$ splits $K\hookrightarrow X$ and $Z\hookrightarrow Y$, so also $f$.

That (ii) implies (iii) is trivial.

Finally, we prove that (iii) implies (i). By \ref{DefSigma} and Theorem~\ref{ThmGro}(ii) and (iii), the category $\Sh\bD$ is a `Grothendieck-tensor category', in the terminology of \cite{CP}. As proved explicitly in \cite{CP}, $\Sh\bD$ is therefore equivalent to the ind-completion of a tensor category if every for every compact $X\in\Sh\bD$, the functor $X\otimes-$ is exact. 
We thus take an arbitrary compact object $X$ in $\Sh\bD$. It follows, using standard properties of compact objects, from Theorem~\ref{ThmGro}(iii) (or see \cite{CP} for the precise statement with proof), that $X$ is the cokernel of a morphism $Y_0(f)$, for $f:B\to A$ in $\bD$. Consider $M\in\Sh\bD$ as in part (iii) which splits $f$. By exactness of $M\otimes-$, it follows that $M\otimes X$ is isomorphic to a direct summand of $M\otimes A$. Hence $M\otimes X\otimes-$ is exact. Since $M\otimes-$ is faithful and exact, also $X\otimes-$ must be exact. 
This concludes the proof.\end{proof}
%%%%%%%%%%%%%%%%%%%%%%%%%%%%%%%%%%%%%%%%%%%%%%%%%%%%%%%%%%%%%%%%%%%%%%%%%%%%%%%%%%%%%%%%%%

\section{Main theorem and applications}\label{SecMain}

\subsection{Main results} Fix a pseudo-tensor category $\bD$ over a field $k$

\begin{theorem}\label{Thm}
Assume that {\underline{one}} of the following conditions is satisfied:
\begin{enumerate}[label=(\roman*)]
\item For every morphism $f$ in $\bD$, there exists $M\in \Sh\bD$, with $M\otimes-:\Sh\bD\to\Sh\bD$ faithful and exact, such that $M\otimes f$ is split in $\Sh\bD$.
%\item For every morphism $f$ in $\bD$, there exists in $\PSh\bD$ a filtered colimit $X$ of strongly faithful objects in $\bD$, such that $X\otimes f$ is split.
\item Every morphism $f$ in $\bD$ is split by a strongly faithful object in $\bD$.
\end{enumerate}
Then $\bD$ admits an abelian envelope $(F,\bT)$. Moreover, there is a tensor equivalence
$\Ind\bT\simeq\Sh\bD$, which admits a commutative (up to isomorphism) diagram of tensor functors
$$\xymatrix{
\Sh\bD\ar[r]^{\sim}&\Ind\bT\\
\bD\ar@{^{(}->}[u]^{Y_0}\ar@{^{(}->}[r]^F&\bT.\ar@{^{(}->}[u]
}$$ 
\end{theorem}
\begin{proof}
We claim that condition (ii) implies condition (i). Indeed, for $X\in\bD$, the fact that $X\otimes-$ is exact follows from bi-adjunction with $X^\vee\otimes-$. Moreover, by Theorem~\ref{ThmGro}(v) if $X$ is strongly faithful then $\ev_X$ is an epimorphism in $\Sh\bD$. For every $A\in\Sh\bD$ we thus have an epimorphism $X^\vee\otimes X\otimes A\tto A$. So $X\otimes A=0$ implies $A=0$, and $X\otimes-$ is faithful.

That condition (i) implies the conclusion is an immediate consequence of Theorem~\ref{ThmShEnv} and Proposition~\ref{PropShInd}.
\end{proof}

\begin{remark}
Theorem~\ref{Thm}(ii) implies in particular that a self-splitting pseudo-tensor category in which every non-zero object is strongly faithful admits an abelian envelope. We cannot remove the latter assumption, since (over any field $k$ of characteristic zero) Lemma~\ref{FunnyEx} provides examples of  self-splitting pseudo-tensor categories which do not admit an abelian envelope  (by Corollary~\ref{CorSF}). If we do not demand that $k$ is algebraically closed, the pseudo-tensor categories can be taken to have finite dimensional morphism spaces.
\end{remark}

\begin{remark}\label{RemRex}
Under assumption \ref{Thm}(ii), one can show directly that the entire class $\Xi(\bD)$ is an `ind-class' as in Definition~\ref{DefInd}. Consequently the subcategory of $\PSh\bD$ which sends all sequences in $\Xi(\bD)$ to exact sequences is a Grothendieck category, by Appendix~\ref{AppTop}, and one can proceed as above to show that it is the ind-completion of the abelian envelope of $\bD$. By uniqueness, it is therefore equivalent to $\Sh\bD$.
Instead, one can also show directly that every functor in $\Sh\bD$ sends every sequence in $\Xi(\bD)$ to an exact sequence in $\Vecc$. Finally, if $\bD$ also happens to be a tensor category, we thus find that $\Sh\bD$ is $\Ind\bD$.
%Note that in an arbitrary pseudo-tensor category $\bD$, for any exact sequence $\xi\in\Xi(\bD)$, the associated sequence $Y_0(\xi)$
%$$\bD(-,X_2)\to\bD(-,X_1)\to\bD(-,X_0)\to 0$$
%is exact in $\Sh\bD$ (but of course not necessarily in $\PSh\bD$). This follows from the Yoneda lemma and the definition of $\Sh\bD$.
\end{remark}

Remark~\ref{RemRex} contains the following well-known observation.
\begin{corollary}\label{CorTriv}
If $\bT$ is a self-splitting tensor category, it is its own abelian envelope.
\end{corollary}

Now we show how our existence result implies the recognition result from \cite{EHS}.

\begin{corollary}[\cite{EHS} Theorem~9.2.2]\label{CorEHS}
Consider a fully faithful tensor functor $I:\bD\to\bV$ to a tensor category $\bV$, such that:
\begin{enumerate}[label=(\roman*)]
\item\label{itms:req_subcat_D2} Any $X\in\bV$ is a quotient of an object $I(A)$, with $A\in\bD$;
\item\label{itms:req_subcat_D3} For any epimorphism $X\to Y$ in $\bV$ there 
exists a nonzero $T\in\bD$
such that $X\otimes I(T)\tto Y\otimes I(T)$ is split;
\end{enumerate}
then $\bV$ is the abelian envelope of $\bD$.
\end{corollary}
\begin{proof}
By Corollary~\ref{CorSF}, every non-zero object in $\bD$ is strongly faithful. Now consider a morphism $a:A\to B$ in $\bD$. Denote its image in $\bV$ by $Z$ and its cokernel by $W$. By Assumption, there exists $0\not=T\in\bD$ such that $T$ splits the epimorphisms $A\tto Z$ and $B\tto W$. It follows that $T\otimes f$ is split. Hence the condition in Theorem~\ref{Thm}(ii) is satisfied.

Thus there exists an abelian envelope $F:\bD\to\bT$. By definition, there exists an exact tensor functor $E:\bT\to\bV$, which extends $I$. By Remark~\ref{RemEF}, $E$ is faithful. Since every object in $\bV$ can be written as the cokernel of a morphism between objects in $\bD\subset\bT$, $E$ is also essentially surjective. By applying the tensor duality, we find also that every object in $\bV$ can be written as the kernel of a morphism between objects in $\bD\subset\bT$. Taking presentations and copresentations of objects in $\bV$ by objects in $\bD$, allows to show that $E$ inherits fully faithfulness from $I$.

In conclusion $E:\bT\to\bV$ is an equivalence, so $\bV$ is the abelian envelope of $\bD$.
\end{proof}

\begin{example}\label{ExamP}
Let $\bT$ be a tensor category which has enough projective objects (or equivalently one non-zero projective object). Then $\bT$ is the abelian envelope of every pseudo-tensor subcategory which contains the projective objects. Indeed, this follows immediately from Corollary~\ref{CorEHS} and Lemma~\ref{LemP}.
\end{example}

\begin{example}\label{ExamG}
Let $\bT$ be a tensor category which has enough projective objects and an object $X$ such that every object in $\bT$ is a subquotient of a direct sum of objects $\otimes^i X\otimes \otimes^jX^\vee$. Then $\bT$ is the abelian envelope of every pseudo-tensor subcategory $\bD\subset\bT$ which contains $X$. Indeed, since every projective object is injective (Lemma~\ref{LemP}), it must appear as a direct summand of a direct sum of objects $\otimes^i X\otimes \otimes^jX^\vee$ and hence be contained in $\bD$. We can thus reduce to Example~\ref{ExamP}
\end{example}
The following example is well known.
\begin{example}\label{ExamGL}
Let $k$ be an algebraically closed field of characteristic zero and $m,n\in\mN$. Denote by $\cK$ the kernel of the tensor functor $H_{m|n}:[\GL_t,k]\to\Rep \GL(m|n)$ from~\ref{FunGL}, for $t=m-n$. By Lemma~\ref{FunGLLem}(i), the functor $H_{m|n}$ is full. The induced functor $[\GL_t,k]/\cK\to\Rep \GL(m|n)$ is thus fully faithful. Every faithful representation of an algebraic supergroup generates the representation category in the sense of Example~\ref{ExamG}, see \cite[\S 7.1]{CH}. Hence $\Rep_k\GL(m|n)$ is the abelian envelope of $[\GL_t,k]/\cK$.
\end{example}

%%%%%%%%%%%%%%%%%%%%%%%%%%%%%%%%%%%%%%%%%%%%%%%%%%%%%%%%%%%%%%%%%%%%%%%%%%%%%%%%%%%%%%%%%%

\subsection{Deligne's categories}
Fix a field $k$ with $\charr(k)=0$ and $\normalt\in\mZ\subset k$. Our results now allow to recover the following theorem of \cite{EHS}.

\begin{theorem}\label{ThmEHS}
\begin{enumerate}[label=(\roman*)]
\item The category $[\GL_\normalt,k]$ has an abelian envelope $\bV_\normalt$.
\item Assume $\overline{k}=k$. Let $\bT$ be a tensor category and take $X\in\bT$ with $\dim X=\normalt$. Either there exists an exact tensor functor
$$\bV_\normalt\;\to\; \bT,\quad \mbox{with }\quad V_\normalt\mapsto X,$$
or there are unique $m,n\in\mN$ with $m-n=\normalt$ for which there exists an exact tensor functor
$$\Rep_k\GL(m|n)\;\to\;\bT,\quad \mbox{with }\quad k^{m|n}\mapsto X.$$
\end{enumerate}
\end{theorem}
\begin{proof}
Part (i) is an immediate application of Theorem~\ref{Thm}(ii), by Theorem~\ref{ThmBrauer}.

Set $\bD:=[\GL_\normalt,k]$. For part (ii) we start from a tensor functor $F:\bD\to\bT$ which maps $V_\normalt$ to $X$, which is guaranteed to exist by Lemma~\ref{LemUni}. Denote by $\cJ$ the kernel of $F$. Since $F$ is monoidal, this is a tensor ideal. By the classification of tensor ideals in \cite[Theorem~7.2.1]{Selecta}, either `$\cJ =0$' or $\cJ $ is equal to the kernel $\cJ _{m|n}$ of $H_{m|n}$ from~\ref{FunGL} for some $m,n$.

If $\cJ =0$ the functor $F$ is faithful, so by Definition~\ref{DefAbEnv} $F$ extends to an exact tensor functor $\bV_\normalt\to\bT$. If $\cJ=\cJ_{m|n} $, $F$ yields a faithful functor $\bD/\cJ \to \bT$ and the exact tensor functor follows from the fact that $\Rep\GL(m|n)$ is an abelian envelope as in Example~\ref{ExamGL}.
\end{proof}

\begin{remark}
It follows easily from the description of the tensor ideals in $[\GL_\normalt,k]$ in \cite[\S 7.2]{Selecta}, that one can determine from which tensor category in Theorem~\ref{ThmEHS}(ii) the exact tensor functor comes by which Schur functors annihilate $X\in\bT$. This is explained in detail in \cite{EHS}, where it is also demonstrated that Theorem~\ref{ThmEHS} together with the tannakian formalism of \cite{Del90} yields an affirmative answer to \cite[Question 10.18]{Deligne}.
\end{remark}

The proof of Theorem~\ref{ThmEHS}, using the input from Lemmata~\ref{LemUniO} and~\ref{FunOSpLem} and \cite[\S 7.1]{Selecta}, also yields the following analogue.
\begin{theorem}\label{ThmOSP}
\begin{enumerate}[label=(\roman*)]
\item The category $[\OO_\normalt,k]$ has an abelian envelope $\bU_\normalt$.
\item Assume $\overline{k}=k$. Let $\bT$ be a tensor category and $X$ a symmetrically self-dual object of dimension $\normalt$. Either there exists an exact tensor functor
$$\bU_\normalt\;\to\; \bT,\quad \mbox{with }\quad U_\normalt\mapsto X,$$
or there are unique $m,n\in\mN$ with $m-2n=\normalt$ for which there exists an exact tensor functor
$$\Rep_k\OSp(m|2n)\;\to\;\bT,\quad \mbox{with }\quad k^{m|2n}\mapsto X.$$
\end{enumerate}
\end{theorem}

%\begin{example}
%The category $\Sh\Rep(\GL(\normalt),\mC)$ is equivalent to the category constructed in \cite{EHS}. An alternative proof of \cite[Question 10.18]{Deligne} now follows from the classification of ideals in $\Rep(\GL(\normalt),\mC)$ in \cite{Selecta}.
%\end{example}

\subsubsection{}\label{AbEnCO} In \cite[\S 2]{Deligne}, a universal pseudo-tensor category $[\SG_t,k]$ is defined for every $t\in k$, which is a semisimple tensor category when $t\not\in\mN$ by \cite[Th\'eor\`eme~2.18]{Deligne}. In \cite[Proposition~8.18]{Deligne} it is shown that, for $n\in\mN$, the pseudo-tensor category $[\SG_n,k]$ admits a fully faithful tensor $F$ functor into a tensor category $\bT_n$. 
By Corollary~\ref{CorSF}, every non-zero object in $[\SG_n,k]$ is strongly faithful. Furthermore, it is proved in \cite[Lemma~3.11]{CO} that there exists a non-zero object which splits every morphism in $[\SG_n,k]$. Theorem~\ref{Thm} therefore demonstrates that $[\SG_n,k]$ admits an abelian envelope. This recovers one of the main results in \cite{CO}. 

It seems worthwhile to point out the following observations (although the equivalent properties are of course known to be true by~\cite{CO}), which do not rely on \cite[Lemma~3.11]{CO}. The latter lemma is one of the cornerstones in both the original and above proof that $[\SG_n,k]$ admits an abelian envelope, but has a rather intricate proof.

\begin{prop}\label{PropNo311}
For $F:[\SG_n,k]\to \bT_n$ in \ref{AbEnCO}, the following properties are equivalent:
\begin{enumerate}[label=(\roman*)]
\item Every object in $\bT_n$ is a quotient of an object $F(D)$ with $D\in [\SG_n,k]$.
\item For every indecomposable $D\in [\SG_n,k]$ with $\dim D=0$, $F(D)$ is projective in $\bT_n$.
\end{enumerate}
Each statement implies that $\bT_n$ is the abelian envelope of $[\SG_n,k]$.
\end{prop}
\begin{proof}
We set $\bD=[\SG_n,k]$.
First we prove that (i) implies (ii). By construction of $\bT_n$ in \cite[\S 2]{Deligne}, all objects have finite length and morphism spaces are finite dimensional. It follows that every object in $\Ind\bT$ is the union of its subobjects in $\bT_n$ and that $\unit$ admits an injective hull $I$ in $\Ind\bT$. So $I$ is the union of objects $I_\alpha\supset \unit$ in $\bT_n$. Moreover, each $I_\alpha$ is a subobject of an object in $\bD$.
By \cite[\S 3.4]{Selecta}, there exists a unique indecomposable object $X_0$ in $\bD$, different from $\unit$, for which there exist non-zero morphisms $\unit\to X_0$ and moreover $\bD(\unit,X_0)=k$. This shows that $X_0$ is in fact the injective hull of $\unit$, so in particular $X_0$ is projective in $\bT_n$. Also by \cite[\S 3.4]{Selecta}, $\dim X_0=0$ and every other indecomposable object in $\bD$ of dimension zero is a direct summand of a tensor product of $X_0$ with some $Z\in\bD$. Thus (i) implies (ii).

Now we prove that (ii) implies (i). By \cite[Proposition~B1]{Deligne}, every object in $\bT_n$ is a subquotient of an object in $\bD$. Now consider an arbitrary $X\in\bT_n$. It is a subquotient of $M\in\bD$. By assumption there exists a projective (and hence injective) object $P$ in $\bT_n$ contained in $\bD$. It then follows that $P\otimes X$ is a direct summand of $P\otimes M$. On the other hand, $X$ is a quotient of $P^\vee\otimes P\otimes X$, which is itself a direct summand of $P^\vee\otimes P\otimes M\in\bD$. So (i) follows.

The combination of (i) and (ii) imply that $\bT_n$ is the abelian envelope of $\bD$, for instance by Corollary~\ref{CorEHS}.
\end{proof}

\begin{remark}
We can also prove that \ref{PropNo311}(i) implies \ref{PropNo311}(ii) by using the observation from \cite{CO} that $[\SG_n,k]$ contains trivial blocks.
\end{remark}

\subsection{Tilting modules}
Now let $k$ be an algebraically closed field of characteristic $p>0$.

\subsubsection{}\label{SetupSL2} We work in the tensor category $\Rep SL_2$ of finite dimensional algebraic representations of the algebraic group $SL_2/k$. 
 We have the pseudo-tensor subcategory $\bD:=\Tilt SL_2$ of tilting modules, see \cite[\S II.E]{Jantzen}. 
 We denote the simple module and the indecomposable tilting module with highest weight $i\omega$ (with $\omega$ the fundamental weight) by $L_i$ and $T_i$, for $i\in\mN$. The 
 Steinberg modules, see \cite[II.3.18]{Jantzen}, are
 $$\St_j\;=\;L_{p^j-1}\;=\; T_{p^j-1},\qquad\mbox{for } \; j\in\mN.$$
 For $r\in\mZ_{>0}$, we consider the tensor ideal $\cJ_r$ in $\Tilt SL_2$ of morphisms which factor through a direct sum of objects $T_i$, with $i\ge (p^r-1)$. This gives a complete and irredundant list of the non-trivial tensor ideals in $\Tilt SL_2$, see~\cite[\S 5.3]{Selecta}. Consequently, $\cJ_r$ is generated by $\id_{\St_{r}}$.

\begin{theorem}\label{ThmBEO}
If $p>2$, then $(\Tilt SL_2)/\cJ_r$ admits an abelian envelope, for each $r>0$.
\end{theorem}
The condition $p>2$ is not required and only reflects the limitations of the proof of Lemma~\ref{LemSt2} below.  Indeed, the equivalent of Theorem~\ref{ThmBEO} for $p=2$ is already known by~\cite{BE}. We start the proof with the following lemma.

\begin{lemma}\label{LemAlc}
If $L_a$ is in the same block of $\Rep SL_2$ as $\St_j=L_{p^j-1}$, for $a,j\in\mN$, then either $a=p^j-1$ or $a\ge2p^{j+1}-p^{j}-1$.
\end{lemma}
\begin{proof}
This is an immediate consequence of \cite[II.7.2(3)]{Jantzen}.
\end{proof}

\begin{lemma}\label{LemSt1}
If $i\le  p^r-1$, then $L_i\otimes \St_{r-1}$ is a tilting module.
\end{lemma}
\begin{proof}
By the Steinberg tensor product theorem, \cite[II.3.17]{Jantzen}, for $i<p^r$ we have
$$L_i\;\simeq\; \bigotimes_{a=0}^{r-1}L_{p^ai_a},\quad \mbox{with}\quad i=\sum_{a=0}^{r-1} p^a i_a\quad\mbox{and}\quad 0\le i_a<p.$$
By Lemma~\ref{LemTriv}, it therefore suffices to prove that $L_{p^ab}\otimes \St_{r-1}$ is a tilting module for $a<r$ and $b<p$. We prove the more general claim that $L_m\otimes \St_{r-1}$ is a tilting module for $m\le p^r-p^{r-1}.$
By \cite[Proposition~E.1]{Jantzen}, it then suffices to prove that
\begin{equation}\label{ExtVan}\Ext^1(\Delta_n,L_m\otimes\St_{r-1} )=0,\qquad\mbox{for $n\in\mN$ and $m\le p^r-p^{r-1}$},\end{equation}
where $\Delta_n$ is the Weyl module with top $L_n$. 

We divide \eqref{ExtVan} into two cases. First assume that $n\ge p^r-1$. Then $n\ge m+p^{r-1}-1$, so $L_m\otimes\St_{r-1}$ belongs to the Serre subcategory ${\Rep SL_2}^{\le n}$ generated by simples $L_j$ with $j\le n$ in which $\Delta_n$ is projective. Hence \eqref{ExtVan} is satisfied. Now assume that $n< p^r-1$. The left-hand of \eqref{ExtVan} can be rewritten as $\Ext^1(\Delta_n\otimes L_m,\St_{r-1} )$, and by our assumption
$$n+m<2p^r-p^{r-1}-1.$$
By Lemma~\ref{LemAlc}, this means that the direct summand of $\Delta_n\otimes L_m$ in the block of $\St_{r-1}$ is a direct sum of copies of $\St_{r-1}$, so the extension vanishes and \eqref{ExtVan} is again satisfied.
\end{proof}

We set $\bD=\Tilt SL_2$ and $\bC=\bD/\cJ_r$.

\begin{lemma}\label{LemSt2}
If $p>2$, the object $\St_{r-1}$ is strongly faithful in $\bC$.
\end{lemma}
\begin{proof}
By definition, we need to prove that the sequence
\begin{equation}\label{seqTilt}\bC(\gamma_{\St_{r-1}}, T_i):\quad  0\to \bC(\unit, T_i)\to \bC(\otimes^2\St_{r-1}, T_i)\to \bC(\otimes^4\St_{r-1}, T_i)\end{equation}
is exact, for each $0\le i<p^r-1$.

The structure of tensor ideals recalled in~\ref{SetupSL2} implies that for $i\ge p^{r-1}-1$, the module $T_i$ is a direct summand of an object $T\otimes \St_{r-1}$.
That \eqref{seqTilt} is exact for $i\ge p^{r-1}-1$ is thus an example of Lemma~\ref{LemEO}.

Next, we consider $T_i$ with $i< p^{r-1}-1$. We claim that
$$\cJ_r(\unit,T_i)=0=\cJ_r(\otimes^2\St_{r-1}, T_i)=\cJ_r(\otimes^4\St_{r-1}, T_i).$$
That the left-most space is zero follows immediately from the description of the ideals $\cJ_l$ in \cite[\S 3.2]{Selecta}. We now prove the claim for the right-most space, the proof for the middle space is similar but easier. Note also that when $p>3$, the proof below even works for $i<p^r-1$, so in that case we do not need the previous paragraph.

By adjunction, we can equivalently prove 
$$\cJ_r(\St_{r-1}, \otimes^3\St_{r-1}\otimes T_i)=0.$$
By definition of $\cJ_r$ and Lemma~\ref{LemAlc}, the contrary would necessarily imply that 
$$[\otimes^3\St_{r-1}\otimes T_i:L_a]\not=0,\qquad\mbox{for some }\; a\ge 2p^r-p^{r-1}-1.$$
However, since we have
$$ i+3(p^{r-1}-1)\;<\; 4p^{r-1} -4\; <\;  2p^r-p^{r-1}-1,$$
under the assumption $p> 2$, this non-vanishing multiplicity is impossible.
It follows that for $i<p^{r-1}-1$ we have a commutative diagram, with the second row given by \eqref{seqTilt}:
$$\xymatrix{
0\ar[r]& \bD(\unit, T_i)\ar[r]\ar[d]^{\sim}&\bD(\otimes^2\St_{r-1}, T_i)\ar[r]\ar[d]^{\sim}& \bD(\otimes^4\St_{r-1}, T_i)\ar[d]^{\sim}\\
0\ar[r]& \bC(\unit, T_i)\ar[r]&\bC(\otimes^2\St_{r-1}, T_i)\ar[r]& \bC(\otimes^4\St_{r-1}, T_i).
}$$
The first row is exact by Corollary~\ref{CorSF} and the inclusion $\bD\subset\Rep SL_2$. Hence the second row is exact. This concludes the proof.
\end{proof}

\begin{proof}[Proof of Theorem~\ref{ThmBEO}]
Consider a morphism $f:T\to T'$ in $\bD=\Tilt SL_2$, where $T$ and $T'$ are direct sums of indecomposable tilting modules $T_i$ with $i< p^r-1$. By Lemma~\ref{LemSt1} the image, kernel and cokernel of $f$ are objects $X\in \Rep SL_2$ such that $\St_{r-1}\otimes X$ is a tilting module. Indeed, this follows from the fact that there are no first extensions between tilting modules, see \cite[\S II.E]{Jantzen}. The same fact then also shows that $\St_{r-1}\otimes f$ is split in $\bD$, see also~\cite{CEH}. It then follows trivially that $\St_{r-1}\otimes f$ is also split in $\bC=\bD/\cJ_r$. Any morphism in $\bC$ can be written as above. Hence $\St_{r-1}$ splits every morphism in $\bC$.

Since $\St_{r-1}$ is strongly faithful in $\bC$, by Lemma~\ref{LemSt2}, we can apply Theorem~\ref{Thm}(ii).
\end{proof}

\begin{remark}\label{RemG1}
Let $G$ be a simple simply-connected algebraic group. The category $\Rep G$ is self-splitting via the Steinberg modules, see \cite[\S 3.3]{CEH}. 
This thus gives an example of a self-splitting tensor category which is not a finite tensor category, and $\Rep G$ is its own abelian envelope by Corollary~\ref{CorTriv}.
%By Theorem~\ref{Thm}(ii), $\Rep G$ admits an abelian envelope and moreover by Corollary~\ref{CorEHS} it is its own abelian envelope. The latter can also be observed directly from Remark~\ref{RemRex}, which states that the ind-completion of the abelian envelope of $\Rep G$ is the category of left exact functors $(\Rep G)^{\op}\to\Vecc$. The latter is of course indeed just $\Ind\Rep G$.
\end{remark}

\begin{remark}
As proved in \cite[Theorem~3.3.1]{CEH}, in the generality of Remark~\ref{RemG1}, $\Rep G$ is the abelian envelope of $\Tilt G$. Let $\Rep^\infty G$ denote the category of all algebraic representations (which is equivalent to $\Ind\Rep G$). Our results can be used to prove that $\Rep^\infty G$ is equivalent to the category of $k$-linear functors $(\Tilt G)^{\op}\to\Vecc$ which send all sequences in $\Xi(\Tilt G)$ (or alternatively all sequences $T\otimes\gamma_{\St_n}$ for tilting modules $T$ and $n\in\mN$) to exact sequences.
\end{remark}

\subsection{A peculiar connection with tensor ideals}

 Fix a pseudo-tensor category $\bD$ over a field $k$ and assume that the {\em morphism spaces in $\bD$ are finite dimensional.}
\subsubsection{} A {\bf thick tensor ideal} in $\bD$ is a full Karoubi subcategory $J$ of $\bD$ such that $X\in J$ implies that $Y\otimes X\in J$ for all $Y\in\bD$. The {\bf decategorification map}, see \cite[\S 4.1]{Selecta}, sends a tensor ideal $\cJ$ in $\bD$ to the thick tensor ideal of objects $X$ with $\id_X\in\cJ$. By \cite[Theorem~4.1.2]{Selecta}, this map is always surjective.

\begin{prop}\label{PropIdeal}
Assume that the decategorification map is a bijection for $\bD$ and that there exists a fully faithful tensor functor $I:\bD\to\bV$ to a tensor category $\bV$, such that
any $X\in\bV$ is a quotient of an object $I(A)$, with $A\in\bD$.
Then $\bV$ is the abelian envelope of $\bD$.
\end{prop}
\begin{proof}
By Corollary~\ref{CorEHS}, it suffices to prove that any epimorphism in $\bV$ is split by a non-zero object in $\bD$. We do this in three steps.

1: Consider a non-zero morphism $f:D\to\unit$ in $\bD$. This is automatically an epimorphism in~$\bV$. By \cite[Proposition~4.2.2]{Selecta}, $f$ is unique up to composition with endomorphisms of $D$. Furthermore, \cite[Lemma~4.2.4]{Selecta} then implies that there exists $X\in\bD$ such that $\ev_X$ is given by a composition
$$X^\vee\otimes X\;\to\; \oplus_{i=1}^n D\; \xrightarrow{(f\circ\phi_i)}\;\unit,$$
for certain $\phi_i\in \End(D)$. We can rewrite this as
$$\ev_X:\;X^\vee\otimes X\;\to \; D\; \xrightarrow{f}\;\unit.$$
By Lemma~\ref{LemSplit}(i) the morphism $X\otimes \ev_X$, and hence also $X\otimes f$ is split.

2: Consider an epimorphism $g:M\tto \unit$ in $\bV$. By assumption, there exists $D\in\bD$ such that we have an epimorphism $\pi:D\tto M$. By step 1, $X\otimes (g\circ\pi)$ is split for some non-zero $X\in\bD$ from which it follows that also $X\otimes g$ is split.

3: Finally, we consider an arbitrary epimorphism $h:M\tto N$ in $\bV$. Tensoring with $N^\vee$ and taking a pullback yields a commutative diagram
$$\xymatrix{
M\otimes N^\vee\ar@{->>}[rr]^{h\otimes N^\vee} && N\otimes N^\vee\\
(M\otimes N^\vee)\times_{(N\otimes N^\vee)}\unit\ar@{->>}[rr]\ar@{^{(}->}[u] &&\unit\ar@{^{(}->}[u]^{\co_N}.
}$$
By step 2, there exists a non-zero $X\in\bD$ which splits the epimorphism on the lower line. After applying $X\otimes-$, the diagram thus admits a diagonal morphism $X\to XMN^\vee$ which makes the upper triangle commute. It then follows that the associated morphism $XN\to XM$ ensures that $X\otimes h$ is split.
\end{proof}

\begin{remark}Let $k$ be algebraically closed.
It is proved in \cite{Selecta} that the decategorification map is a bijection for $[\GL_\normalt,k]$, $[\OO_\normalt,k]$ and $[\mathsf{S}_\normalt,k]$, when $\charr(k)=0$, and that the same is true for $\Tilt SL_2$ when $\charr(k)>0$. \end{remark}
\appendix

\section{Grothendieck topologies}\label{AppTop}
Fix a commutative ring $K$ and an essentially small $K$-linear category $\bA$ for the entire appendix. Denote by $\PSh\bA$ the category of presheaves $\bA^{\op}\to K-\Mod$.
\subsection{$K$-linear sheaves}
\subsubsection{}
For $A\in\bA$, a {\bf sieve on} $A$ is a $K$-linear subfunctor of $\bA(-,A)\in\PSh\bA$. For a sieve $R$ on $A$ and a morphism $f:B\to A$ in $\bA$, the assignent
$$\Ob \bA\to K-\Mod,\;\, C\mapsto \{g\in \bA(C,B)\,|\, f\circ g\in R(C)\},$$
yields a sieve on $B$, which we denote by $f^{-1}R$. In other words, $f^{-1}R$ is the pullback of $R\to \bA(-,A)\leftarrow \bA(-,B)$.

The following definition is taken from \cite[1.2 and 1.6]{BQ}. 
\begin{definition}\label{Def1} A $K$-linear Grothendieck topology $\cT$ on $\bA$ is an assignment to each $A\in\bA$ of a collection $\cT(A)$ of sieves on $A$ such that for every $A\in\bA$:
\begin{enumerate}
\item[(T1)] We have $\bA(-,A)\in \cT(A)$;
\item[(T2)] For $R\in \cT(A)$ and a morphism $f:B\to A$  in $\bA$, we have $f^{-1}R\in \cT(B)$;
\item[(T3)] For a sieve $S$ on $A$ and $R\in\cT(A)$ such that for every $B\in\bA$ and $f\in R(B)\subset \bA(B,A)$ we have $f^{-1}S \in \cT(B)$, it follows that $S\in\cT(A)$.
\end{enumerate}
\end{definition}

The following definition is taken from \cite[1.3 and 1.6]{BQ}.
\begin{definition}\label{Def2}
For a $K$-linear Grothendieck topology $\cT$ on $\bA$, a presheaf $F\in\PSh\bA$ is a {\bf $\cT$-sheaf} if for every $A\in\bA$ and $R\in\cT(A)$, the canonical morphism
$$F(A)\simeq \Nat(\bA(-,A),F)\to \Nat(R,F)$$
is an isomorphism. The full subcategory of $\PSh\bA$ of $\cT$-sheaves is denoted by $\Sh(\bA,\cT)$.
\end{definition}

Our interest in Grothendieck topologies derives from \cite[Theorem~1.5]{BQ}. Recall that a localisation of an abelian category is a full replete subcategory for which the inclusion functor has a left adjoint which is left exact (and hence exact).

\begin{theorem}[Borceux - Quinteiro]\label{ThmBQ}
The localisations of $\PSh\bA$ are precisely the subcategories $\Sh(\bA,\cT)$, for all Grothendieck topologies $\cT$ on $\bA$.
\end{theorem}

\subsection{Sch\"appi's formalism}\label{AppSch}

We start by recalling a definition from \cite{Sch2}.

\begin{definition}\label{DefInd}
For a class $\Sigma\subset\Xi(\bA)$ of exact sequences 
\eqref{rexeq},
denote by $\Co(\Sigma)$ the set of morphisms $q$ which appear as the cokernels in sequences in $\Sigma$. Then $\Sigma$ is an {\bf ind-class} if
\begin{enumerate}[label=(\roman*)]
\item For every $q\in\Co(\Sigma)$, there is a sequence $X_1\stackrel{q}{\to} X_0\to Z\stackrel{\sim}{\to} 0 $ in $\Sigma$.
\item For each sequence \eqref{rexeq} in $\Sigma$ and each morphism $f:A\to X_1$ in $\bA$ with $q\circ f=0$, there exists $p':B\to A$ in $ \Co(\Sigma)$ and $f':B\to X_2$ in $\bA$ yielding a commutative diagram
$$\xymatrix{
X_2\ar[r]^p& X_1\ar[r]^q& X_0\ar[r]&0\\
B\ar[r]^{p'}\ar[u]^{f'}&A\ar[u]^f\ar@{-}[ru]_{0}.
}$$
\end{enumerate}
\end{definition}

The following proposition follows immediately from \cite[A.1.2 and A.2.3]{Sch2}.
\begin{prop}[Sch\"appi]\label{LemSch}
Consider a subclass $\Sigma\subset\Xi(\bA)$. For each $A\in\bA$, denote by $\cT(A)$ the set of all sieves $R\subset \bA(-,A)$ which contain a composite $r=r_1\circ r_2\circ\cdots\circ r_m$ (with $m\in\mN$, where the empty composite is interpreted as $\id_A$) of morphisms $r_i\in\Co(\Sigma)$. 

If $\Sigma$ is an ind-class, then $\{A\mapsto\cT(A)\}$ is a $K$-linear Grothendieck topology on $\bA$.
\end{prop}

The following proposition follows from the combination of \cite[A.1.4 and A.2.5]{Sch2}.
\begin{prop}[Sch\"appi]\label{LemSch2}
For the topology $\cT$ associated to an ind-class $\Sigma\subset\Xi(\bA)$ as in Lemma~\ref{LemSch} and $F\in\PSh\bA$, the following are equivalent
\begin{enumerate}[label=(\roman*)]
\item $F$ is a $\cT$-sheaf;
\item The sequence $F(\xi)$ is exact in $K-\Mod$ for each $\xi\in\Sigma$.
\end{enumerate}
\end{prop}
%\begin{proof}
%The equivalence of (i) and (ii) is a special case of . The equivalence of (ii) and (iii) is precisely~\cite[Proposition~2.6]{Sch1}.
%\end{proof}

% For each $q:X_1\to X_0$ in $\Co(\Sigma)$, we have an exact sequence
%$$0\to F(X_0)\stackrel{F(q)}{\to} F(X_1)\to \prod_{f:A\to X_1, q\circ f=0}F(A);$$

\subsection*{Acknowledgement}
The research was supported by ARC grant DP180102563. The author thanks Inna Entova, Pavel Etingof, Thorsten Heidersdorf, Bregje Pauwels and Victor Ostrik for interesting discussions and in particular Pavel Etingof and Victor Ostrik for suggesting improvements of Lemma~\ref{LemSt2} allowing to include the case $p=3$.

\end{document}